\documentclass[reqno,english]{amsart}
\usepackage[T1]{fontenc}
\usepackage[utf8]{inputenc}
\usepackage{lmodern,amsmath,amssymb,amsthm,mathtools,mathrsfs,microtype,booktabs,array}
\usepackage{needspace}
\usepackage{etoolbox}
\usepackage{subcaption}
\usepackage[normalem]{ulem}
\usepackage[export]{adjustbox}
\usepackage{esint}
\usepackage{xcolor}
\usepackage{array}
\usepackage{esint}
\newcommand{\Rea}{\operatorname{Re}}
\newcommand{\dist}{\operatorname{dist}}
\usepackage{xcolor}
\usepackage{array}
\usepackage[colorlinks=true]{hyperref}
\hypersetup{urlcolor=blue, citecolor=blue, linkcolor=red}

\newtheorem{theorem}{Theorem}
\newtheorem{lemma}{Lemma}[section]
\newtheorem{proposition}[lemma]{Proposition}
\newtheorem{corollary}[lemma]{Corollary}
\theoremstyle{remark}\newtheorem{remark}[lemma]{Remark}
\usepackage{mathtools}
\usepackage{thmtools}
\declaretheoremstyle[headfont=\normalfont]{normalhead}

\numberwithin{equation}{section}

\newtheorem{othertheorem}{Theorem}

\newcommand{\Ent}{\operatorname{Ent}}
\newcommand{\NN}{\mathcal N}
\newcommand{\Log}{\operatorname{Log}}
\newcommand{\Acal}{\mathcal A}
\newcommand{\dd}{\,\mathrm{d}}

\title[Sharp CR moment inequalities]{Sharp CR Beckner--Onofri inequalities\\
	under high-order holomorphic moment constraints}

\author{Tuoxin Li}
\address{Department of Mathematics, University of Macau, Taipa, Macau}
\email{tuoxinli@um.edu.mo}

\author{Jingxuan Sun}
\address{Department of Mathematics, University of Macau, Taipa, Macau}
\email{jingxuansun@um.edu.mo}

\date{}

\subjclass[2020]{Primary 26D10, 32V20; Secondary 32A35, 35B38}

\keywords{Beckner--Onofri inequality, CR sphere, holomorphic moments, entropy}
\begin{document}

	\begin{abstract}
			We establish sharp Beckner--Onofri on the CR sphere $\mathbb{S}^{2N-1}\subset\mathbb C^N$, $N\ge2$, under vanishing holomorphic moments through degree $m\ge1$. The optimal coefficients are $\frac{1}{m+1}$. The function inequality admits a stronger form involving only the energy above degree $m$. We also prove a quantitative concentration estimate in the $1$-Wasserstein distance and show that every maximizing sequence has diverging energy and approaches the family of uniform measures on regular $(m+1)$-polygons. The proof combines subcritical minimization with rigidity of the constrained Euler--Lagrange equation, obtained from an exact radial energy identity and a Hardy space projection estimate. Finally, weighted ball automorphisms yield sharp inequalities without moment constraints, with scalar conformal, barycenter, and holomorphic Gram determinant corrections.
	\end{abstract}
	\maketitle

	\section{Introduction}\label{sec:introduction}
	
	Sharp exponential inequalities connect critical Sobolev embeddings, conformal geometry, and the concentration of probability measures. We first recall the Lebedev--Milin inequality on the circle, which comes from the classical complex analysis.
	
	\begin{othertheorem}[Lebedev--Milin {\cite{LebedevMilin1965}}]
		\label{thm:intro-lebedev-milin}
		Let $F(\zeta)=\sum_{k\ge0}a_k\zeta^k$ be holomorphic on the unit disk $\mathbb D$, with $\sum_{k\ge1}k|a_k|^2<\infty$. Then 
		\begin{equation}\label{eq:intro-lebedev-milin}
			\log\left(\frac{1}{2\pi}\int_{0}^{2\pi}
			e^{2\Rea(F(e^{i\theta})-a_0)}\dd\theta\right)
			\le \sum_{k\ge1}k|a_k|^2
			=\frac{1}{\pi}\int_{\mathbb D}|F'(\zeta)|^2\dd A(\zeta),
		\end{equation}
		where $A$ is Euclidean area measure. Equality holds precisely for	$F(\zeta)=a_0-\Log(1-\bar b\zeta)$, $b\in\mathbb D$, with the logarithm vanishing at the origin.
	\end{othertheorem}
	We refer readers to \cite[Chapter~5]{Duren1983} for a standard account.
	
	There is also an higher-moment refinement. If, in addition,
	\begin{equation*}
		\int_{0}^{2\pi}e^{ij\theta}e^{2\Rea F(e^{i\theta})}\dd\theta=0,
		\qquad 1\le j\le m,
	\end{equation*}
	then the right-hand side of \eqref{eq:intro-lebedev-milin} can be replaced by $\frac{1}{m+1}\sum_{k\ge1}k|a_k|^2$.	The case $m=1$ is due to Osgood, Phillips and Sarnak	\cite{OPS1988}; Widom obtained the general refinement from Toeplitz determinant inequalities \cite{Widom1988}; see also \cite{ChangHang2022}. Holomorphic norm inequalities and their reproducing-kernel formulations were also developed by Saitoh and Burbea \cite{Burbea1982,Burbea1983,Saitoh1979}.

	There are many generalizations of \eqref{eq:intro-lebedev-milin} on higher-dimensional spheres.	The Moser--Trudinger inequality \cite{Moser1971} and Onofri's sharp inequality on $\mathbb S^2$ \cite{Onofri1982} are basic examples. Later, Beckner extended the conformal inequality to higher-dimensional	spheres.
	\begin{othertheorem}[\cite{Beckner1993}]
		Let $\dd w$ be the probability surface measure on $\mathbb{S}^N$. Then for any $u\in H^{\frac{N}{2}}(\mathbb{S}^N)$,  
		\begin{equation}\label{Beckner}
			\int_{\mathbb{S}^N}u(P_{N}u)\dd w+2(N-1)!\int_{\mathbb{S}^N}u	\dd w-\frac{2(N-1)!}{N}\log\int_{\mathbb{S}^N}e^{Nu}\dd w\geq 0,
		\end{equation}
		where $P_N$ is the critical conformal operator on $\mathbb{S}^N$. Equality holds if and only if
		\begin{equation*}
			u(x)=c-\log(1-\xi\cdot x),\qquad |\xi|<1.
		\end{equation*}
	\end{othertheorem}
In view of the refinement of the Lebedov--Milin inequality, these results naturally lead to the question of how much the sharp constant improves when the associated exponential measure satisfies additional moment conditions.
	
	Aubin \cite{Aubin1979} observed that a vanishing first moment excludes concentration at one point and improves the leading coefficient. This improvement plays an important role in the	prescribed curvature problem and in the work of Chang and Yang	\cite{ChangYang1987,ChangYang1988}. Then conjectured that if
	\begin{equation*}
		\int_{\mathbb{S}^N}e^{Nu} x_j \dd w=0,\quad j=1,\dots, N+1,
	\end{equation*}
	Then the coefficient of of first term of \eqref{Beckner} can be improved to $\frac{1}{2}$. On $\mathbb S^2$, this was proved for axially symmetric functions by Gui and Wei \cite{GuiWei2000} and in general by Gui and Moradifam \cite{GuiMoradifam2018}. Gui, Li, Wei and Ye \cite{GLWYfull2026} subsequently established the corresponding centered inequality on $\mathbb{S}^N$ for all $N\geq 3$, and the classification of solutions of the corresponding Euler--Lagrange equation was obtained by Gui, Hu, Li, Wei, Xie and Ye \cite{GHLWXY2026}.

	Higher-order moments impose further restrictions on concentration. Chang and Hang related improved Moser--Trudinger--Onofri inequalities to positive cubature formulas and the minimum number of possible	concentration points \cite{ChangHang2022}. Their general	cubature-based estimates allow an arbitrarily small loss in the leading coefficient and an additive constant. Hang \cite{Hang2022} refined the concentration compactness principle in critical dimension and extended this approach to higher-dimensional spheres. Related moment methods appear in singular Liouville equations \cite{BartolucciMalchiodi2013}, improved Sobolev inequalities \cite{HangWang2022} and Sobolev trace inequalities \cite{ChenWeiWu2024}, 	Moser--Trudinger--Onofri inequalities under second-moment	constraints \cite{ChenZhang2025}, constrained	Hardy--Littlewood--Sobolev inequalities \cite{GKSV2026,HuDou2023}, and almost-sharp inequalities \cite{SunZou2025}.
	 However, sharp inequalities under higher-order moment constraints remain open in general. 
	
\medskip

	Another generalization of \eqref{eq:intro-lebedev-milin} is the corresponding inequality on the CR sphere $\mathbb{S}^{2N-1}\subset\mathbb C^N$. Branson, Fontana and Morpurgo \cite{BFM2013} established the sharp	Beckner--Onofri inequality for CR pluriharmonic functions and its	logarithmic Hardy--Littlewood--Sobolev counterpart. They also establish the almost-sharp inequality \cite[Proposition~3.4]{BFM2013} under the centered te constraint. We refer to \cite{FrankLieb2012} about sharp inequalities on the Heisenberg group  and \cite{Ho2023,Yan2023,Yan2025} about more results on constrained CR Sobolev inequalities.
	
	
	In this paper we prove the sharp Beckner--Onofri type inequality on CR spheres under vanishing	holomorphic moments of every prescribed order. We also quantify the concentration that makes the coefficient sharp and derive global inequalities with conformal corrections. 

	\subsection{Sharp inequalities under holomorphic moment constraints}
	
	We now formulate our main result on the CR sphere.	Let $\sigma$ be the probability surface measure on $\mathbb{S}^{2N-1}=\partial\mathbb B^N\subset\mathbb C^N$, where $N\ge2$ and $\mathbb B^N=\{z\in\mathbb C^N:|z|<1\}$.
	We use the bilinear dot product
	\begin{equation*}
		z\cdot w=\sum_{j=1}^{N}z_jw_j,\qquad
		|z|^2=z\cdot\bar z.
	\end{equation*}
	 We use $*$	to denote the adjoint or conjugate transpose. For a positive Borel measure $\mu$ on the sphere, our convention is
	\begin{equation*}
		\langle u,v\rangle_{L^2(\mu)}
		=\int_{\mathbb{S}^{2N-1}}u\bar v\dd\mu.
	\end{equation*}
	All complex inner products are linear in the first argument.	We write $\langle u,v\rangle_2$ and $\|u\|_2$ for $\mu=\sigma$.
	
	Set
	\begin{equation}\label{def Dk lak}
		D_k=\binom{N+k-1}{k},\qquad \lambda_k=\frac{kD_k}{N}.
	\end{equation}
	For a holomorphic function $f$ on $\mathbb B^N$, we write $f=\sum_{k\ge0}f_k$, where $f_k$ is homogeneous of degree $k$. We define the energy
	\begin{equation}\label{eq:energy}
		E_N(f)=\sum_{k\ge1}\frac{kD_k}{N}\|f_k\|_2^2,
		\qquad
		E_{>m}(f)=\sum_{k>m}\frac{kD_k}{N}\|f_k\|_2^2.
	\end{equation}
	Here $f_0=f(0)$. Since $\frac{kD_k}{N}\ge1$ for $k\ge1$, finite energy gives an $L^2$ boundary trace.
	
	These weights arise from the critical CR quadratic form. The CR sphere has homogeneous dimension $Q=2N$. The critical intertwining operator $A_Q$ annihilates CR pluriharmonic functions, whose real-valued representatives here are $U=f+\bar f$. Its renormalized limit on this subspace is the	\emph{conditional intertwinor} $A'_Q$; see	\cite[Section~1, Proposition~1.2]{BFM2013}.	It annihilates constants and acts on holomorphic and	antiholomorphic terms of degree $k\ge1$ by a multiplication by
	\begin{equation*}
		\frac{\Gamma(k+N)}{\Gamma(k)}
		=k(k+1)\cdots(k+N-1)
		=N!\lambda_k.
	\end{equation*}
	The nonconstant holomorphic and antiholomorphic terms are
	orthogonal in $L^2(\sigma)$. Hence, for $U=f+\bar f$,
	\begin{equation}\label{eq:intro-cr-normalization}
		E_N(f)=\frac{1}{2N!}\int_{\mathbb{S}^{2N-1}} UA'_QU\dd\sigma,
	\end{equation}
	understood as a quadratic form for finite-energy traces. For $N=1$ the same definition reduces to the Dirichlet energy in \eqref{eq:intro-lebedev-milin}.
	
	The scaled version of the ordinary CR inequality takes the following form.
	\begin{othertheorem}[Branson--Fontana--Morpurgo
		{\cite[Theorem~3.1]{BFM2013}}]\label{thm:intro-bfm}
		For every holomorphic $f$ with $E_N(f)<\infty$,
		\begin{equation}\label{eq:ordinary}
			\int_{\mathbb{S}^{2N-1}} e^{2t\Rea(f-f_0)}\dd\sigma
			\le e^{t^2E_N(f)},
			\qquad t\in\mathbb R,\quad E_N(f)<\infty.
		\end{equation}
		For $t=1$, equality holds if and only if
		$f(z)=c-N\Log(1-z\cdot\bar a)$, with
		$c\in\mathbb C$ and $a\in\mathbb B^N$.
		Here the holomorphic logarithm vanishes at the origin.
	\end{othertheorem}
	
	Define
	\begin{equation}\label{eq:pressure-def}
		Z_f=\int_{\mathbb{S}^{2N-1}} e^{f+\bar f}\dd\sigma,\qquad
		P(f)=\log Z_f-2\Rea f_0,\qquad
		\rho_f=Z_f^{-1}e^{f+\bar f}.
	\end{equation}
	In particular, \eqref{eq:ordinary} gives
	$P(f)\le E_N(f)$ and $\rho_f\in L^2$
	whenever $E_N(f)<\infty$.
	
	For multi-indices $\gamma\in\mathbb Z_{\ge0}^N$, we write $|\gamma|=\sum_i\gamma_i$, $\gamma!=\prod_i\gamma_i!$, and $z^\gamma=\prod_i z_i^{\gamma_i}$. For a probability density $\rho$, we define the entropy
	\begin{equation}\label{def ent}
		\Ent(\rho)=\int_{\mathbb{S}^{2N-1}}\rho\log\rho\dd\sigma.
	\end{equation}
	Since $\rho_f\in L^2$, we also have $\Ent(\rho_f)<\infty$ for every finite-energy $f$.
	
	For an integer $m\ge1$, set $q=m+1$ and let
	\begin{equation*}
		\Acal_m=\left\{\rho\ge0:\
		\int_{\mathbb{S}^{2N-1}}\rho\dd\sigma=1,\quad
		\Ent(\rho)<\infty,\quad
		m_\gamma(\rho)=0\
		(1\le|\gamma|\le m)\right\},
	\end{equation*}
	where
	\begin{equation*}
		m_\gamma(\rho)=\int_{\mathbb{S}^{2N-1}} z^\gamma\rho\dd\sigma.
	\end{equation*}
	
		We also introduce the logarithmic interaction
	\begin{equation}\label{eq:interaction-def}
		S[\rho]=-\iint_{\mathbb{S}^{2N-1}\times\mathbb{S}^{2N-1}}\log|1-z\cdot\bar w|\,
		\rho(z)\rho(w)\dd\sigma(z)\dd\sigma(w).
	\end{equation}
	Finite entropy ensures its absolute	integrability, as shown in	Lemma~\ref{lem:poles}.
	
	Then our main theorem can be stated as follows:
	\begin{theorem}\label{thm:pure}
		Let $m\ge1$ be an integer and $q=m+1$. Then every $\rho\in\Acal_m$ satisfies
		\begin{equation}\label{eq:entropy-main}
			\Ent(\rho)\ge qNS[\rho],
		\end{equation}
		with equality if and only if $\rho=1$.
		
		If $E_N(f)<\infty$ and $\rho_f\in\Acal_m$, then
		\begin{equation}\label{eq:pressure-high}
			P(f)\le\frac{1}{q}E_{>m}(f)\le\frac{1}{q}E_N(f).
		\end{equation}
		Equality in either bound for $P(f)$ holds if and only if $f$ is constant.
		
		Moreover, both coefficients in	\eqref{eq:entropy-main} and \eqref{eq:pressure-high} are sharp.
	\end{theorem}

	\begin{remark}
		The conditions on $\rho_f$ do not require $f_1,\ldots,f_m$ to vanish. Thus	\eqref{eq:pressure-high} omits low-degree energy even when those terms are present in $f$. For $m=1$, the constraint is the	usual barycenter condition. However, higher pure holomorphic moments do not prescribe mixed	moments such as	$\int_{\mathbb{S}^{2N-1}} z_i\bar z_j\rho\dd\sigma$.
	\end{remark}
	
	Taking $m=1$, we immediately obtain the following sharp centered inequality.
	\begin{corollary}\label{thm:centered}
		If $E_N(f)<\infty$ and	$\int_{\mathbb{S}^{2N-1}} z\rho_f\dd\sigma=0$, then
		\begin{equation}\label{eq:centered-main}
			P(f)\le\frac{1}{2}E_{>1}(f)			\le\frac{1}{2}E_N(f).
		\end{equation}
		Equality in either bound for $P(f)$ holds if and only if $f$ is constant. The coefficient $\frac{1}{2}$ is optimal.
	\end{corollary}
	\begin{remark}
The sharp inequality under the first-moment constraint was also recently established independently by Chen and Lu \cite{ChenLu2026}.
\end{remark}

	This corollary gives the CR counterpart of the sharp centered inequality on real spheres,	closely related to the Chang--Yang problem \cite{ChangYang1987,ChangYang1988,		GuiMoradifam2018,GLWYfull2026}.
	In particular, it reaches the endpoint coefficient $\frac{1}{2}$ in the CR Beckner--Onofri inequality without the lower-order correction or additive	constant present in the earlier estimates of Branson, Fontana and Morpurgo	\cite[Proposition~3.4]{BFM2013}.

	\subsection{Quantitative polygon concentration}
	
	The sequences establishing sharpness suggest a specific limiting
	geometry: equal masses at the vertices of a regular $q$-polygon
	in a complex line. We show that every ratio-maximizing sequence
	approaches this family. Our estimate controls both the locations
	of the concentrating mass and its distribution among the vertices.
	It is a statement about large-energy concentration; it does not
	identify these atomic measures with the finite-energy equality cases.
	
	Let $\omega=e^{2\pi i/q}$ be a primitive $q$-th root of unity.
	For $v\in\mathbb{S}^{2N-1}$, set
	\begin{equation*}
		\Pi_v=\{\omega^s v:0\le s<q\},\qquad
		\pi_v=\frac{1}{q}\sum_{s=0}^{q-1}\delta_{\omega^s v}.
	\end{equation*}
	For Borel probability measures $\mu,\nu$ on the sphere, the wasserstein distance is defined by
	\begin{equation}\label{def wasser}
		W_1(\mu,\nu)=
		\inf_{\Gamma\in\Pi(\mu,\nu)}
		\int_{\mathbb{S}^{2N-1}\times\mathbb{S}^{2N-1}}
		|z-w|\dd\Gamma(z,w),
	\end{equation}
	Here $\Pi(\mu,\nu)$ is the set of Borel probability measures $\Gamma$ on the product sphere satisfying
	\begin{equation*}
		\Gamma(B\times\mathbb{S}^{2N-1})=\mu(B),\qquad
		\Gamma(\mathbb{S}^{2N-1}\times B)=\nu(B)
	\end{equation*}
	for every Borel set $B\subset\mathbb{S}^{2N-1}$. $W_1$ is the least total transport cost for the Euclidean chord distance $|z-w|$. The infimum is attained, and $W_1$ metrizes weak convergence on this compact sphere; see \cite[Chapters~4 and~6]{Villani2009}. The elementary comparison
	\begin{equation}\label{eq:chord-cap}
		|z-w|^2=2\Rea(1-z\cdot\overline w)
		\le 2|1-z\cdot\overline w|\le4
	\end{equation}
	gives $W_1\le2$. We measure distance to the polygon family by
	\begin{equation}\label{def dq mu}
		d_q(\mu)=\inf_{v\in\mathbb{S}^{2N-1}}W_1(\mu,\pi_v).
	\end{equation}
	In particular, $d_q$ controls the locations of the mass and its division into $q$ equal parts.
	\begin{theorem}\label{thm:quantitative}
		There is a constant $C_{N,q}<\infty$ such that if $\rho\in\mathcal{A}_m$, then
		\begin{equation}\label{eq:density-rate}
			d_q(\rho\dd\sigma)^2
			\le C_{N,q}\frac{\Ent(\rho)-qNS[\rho]+\log(2+\Ent(\rho))}{1+\Ent(\rho)}.
		\end{equation}
		If $E_N(f)<\infty$, $f(0)=0$, and $\rho_f\in\Acal_m$, then
		\begin{equation}\label{eq:source-rate}
			d_q(\rho_f\dd\sigma)^2
			\le C_{N,q}\frac{E_N(f)-qP(f)+\log(2+E_N(f))}{1+E_N(f)}.
		\end{equation}
	\end{theorem}
	
	The next theorem shows that bounded energy is incompatible with	maximizing the quotient $\frac{P(f)}{E_N(f)}$. Together with the quantitative estimate, this identifies every subsequential weak	limit of a maximizing sequence.

	\begin{theorem}\label{thm:maximizing}
		Suppose that a sequence $(f_j)$ satisfies $f_j\ne0$, $E_N(f_j)<\infty$, $f_j(0)=0$,  $\rho_{f_j}\in\mathcal{A}_m$, and
		\begin{equation}\label{eq:ratio-maximizing}
			\frac{P(f_j)}{E_N(f_j)}\longrightarrow\frac{1}{q}.
		\end{equation}
		Then
		\begin{equation}\label{eq:energy-diverges}
			E_N(f_j)\longrightarrow\infty,
			\qquad
			d_q(\rho_{f_j}\dd\sigma)\longrightarrow0.
		\end{equation}
		Every subsequential weak limit of $\rho_{f_j}\dd\sigma$ is therefore
		the uniform probability measure on a regular $q$-polygon. Moreover,
		\begin{equation}\label{eq:entropy-ratios}
			\frac{\Ent(\rho_{f_j})}{E_N(f_j)}\longrightarrow\frac{1}{q},
			\qquad
			\frac{NS[\rho_{f_j}]}{E_N(f_j)}\longrightarrow\frac{1}{q^2}.
		\end{equation}
	\end{theorem}
	
	\medskip
	
	\subsection{Global conformal corrections}
	
	The moment conditions sharpen the inequality but restrict the admissible functions. To obtain a global statement, we use ball automorphisms to center the exponential density, while compensating for the change of energy by a conformal correction. This leads to a scalar inequality without moment constraints and two geometric consequences.
	
	For $s>0$ and $E_N(f)<\infty$, define
	\begin{equation}\label{eq:intro-global-functional}
		\Lambda_s(f)=
		\inf_{a\in\mathbb B^N}
		\frac{\displaystyle\int_{\mathbb{S}^{2N-1}}
			|1-z\cdot\bar a|^{2s}e^{f+\bar f}\dd\sigma}
		{(1-|a|^2)^s}.
	\end{equation}
	This is equivalent to the frame formulation in
	Section~\ref{sec:weighted}, by setting
	$t=\frac{(1,-a^*)}{\sqrt{1-|a|^2}}$.
	We also write
	$M_f=\int_{\mathbb{S}^{2N-1}} z e^{f+\bar f}\dd\sigma$,
	so that $|M_f|<Z_f$.
	
	For each integer $r\ge0$, let
	$\mathscr P_r^{\mathrm{hol}}$ be the space
	of holomorphic polynomials of degree
	at most $r$, with dimension
	$d_r=\binom{N+r}{r}$.
	Given an $L^2(\sigma)$-orthonormal basis
	$(p_i)_{i=1}^{d_r}$, define
	\begin{equation}\label{eq:hol-gram-def}
		G_r(f)_{ij}
		=\int_{\mathbb{S}^{2N-1}} p_i\overline{p_j}e^{f+\bar f}\dd\sigma,
		\qquad
		\Psi_r(f)=\frac{1}{d_r}\log\det G_r(f).
	\end{equation}
	The Gram matrix is positive definite,
	and $\Psi_r(f)$ is independent of the
	chosen basis. Throughout,
	$\Log(1-z\cdot\bar a)$ denotes the
	holomorphic logarithm normalized to
	vanish at the origin.
	
		The case $m=1$ of Theorem~\ref{thm:pure}, combined with the weighted conformal action, yields the following sharp scalar correction without moment constraints.
	
	\begin{theorem}\label{thm:global}
		Let $E_N(f)<\infty$. Then for every $0<s\le N$,
		\begin{equation}\label{eq:complex-power-main}
			\log\Lambda_s(f)
			\le2\Rea f_0+\frac{N}{N+s}E_N(f).
		\end{equation}
		The infimum defining $\Lambda_s(f)$ is attained at a unique $(t_0,t')$ modulo phase. Equality holds if and only if
		\begin{equation}\label{eq:weighted-extremals}
			f(z)=c-(N+s)\Log(1-z\cdot\bar a),
			\qquad c\in\mathbb C,\quad a\in\mathbb B^N.
		\end{equation}
		The coefficient $\frac{N}{N+s}$ is optimal.
		
		Moreover, for every $s>N$, the difference between the left and right sides of		\eqref{eq:complex-power-main} is unbounded above, even among smooth functions satisfying		$\int_{\mathbb{S}^{2N-1}} z\rho_f\dd\sigma=0$.
	\end{theorem}
	
	Two geometric inequalities follow from the scalar correction, the first involving the barycenter of the exponential density.
	
	\begin{corollary}\label{cor:barycenter}
		Let $E_N(f)<\infty$. Then
		\begin{equation}\label{eq:barycenter-main}
			\frac{1}{2}\log(Z_f^2-|M_f|^2)
			\le2\Rea f_0+\frac{2N}{2N+1}E_N(f),
		\end{equation}
		with equality if and only if $f$ is constant. The coefficient $\frac{2N}{2N+1}$ is optimal.
	\end{corollary}
	
	The second involves the Gram determinant of holomorphic polynomials.
	
	\begin{corollary}\label{cor:hol-gram}
		Let $E_N(f)<\infty$. Then for every integer	$1\le r\le N$,
		\begin{equation}\label{eq:hol-gram-main}
			\Psi_r(f)\le2\Rea f_0+\frac{N}{N+r}E_N(f),
		\end{equation}
		with equality if and only if $f$ has the form \eqref{eq:weighted-extremals}		with $s=r$. The coefficient $\frac{N}{N+r}$	is optimal.
	\end{corollary}

	The barycenter inequality is a CR analogue of the sharp real-sphere correction of Chang and Gui \cite{ChangGui2023}; quantitative stability for that inequality was studied by Ghosh and Karmakar \cite{GhoshKarmakar2025}. Related second-moment corrections and constrained mean field equations appear in	\cite{ChenZhang2025,GHX2026}. The Gram determinant estimate belongs to a further line of exponential determinant inequalities	\cite{Berman2009,Fang2004}. Here the conformal covariance of	holomorphic polynomial spaces links the determinant to the scalar correction. These global consequences use the centered case of the main theorem.

	The proof of the sharp moment inequalities proceeds first at the	level of probability densities. Localized entropy estimates give	subcritical minimizers. For their constrained Euler--Lagrange equation, the radial energy of a holomorphic square root of the density satisfies an exact moment identity. A Hardy space	projection estimate yields a strictly stronger lower bound unless	the density is uniform. The key step controls a joint Gram matrix. Rigidity then gives the sharp entropy inequality, and weighted Cauchy--Schwarz	yields the function inequality with all constrained-degree	energies omitted.
	
	The concentration argument combines localized entropy control with a quantitative consequence of the Fej\'er polynomial. We show that a cap with mass close to $\frac{1}{q}$ determines a nearby regular polygon. Nearest-vertex transport controls the locations, and the moment conditions control the vertex weights. A local spectral estimate, compactness, and the equality characterization rule out maximizing sequences with bounded-energy. Finally, a weighted action of ball automorphisms supplies the global conformal corrections.
	
	\medskip
	
	The paper is organized as follows. Section~\ref{sec:pure-proof} proves the sharp moment inequalities, including subcritical minimization, rigidity, and sharpness.	Section~\ref{sec:quantitative-concentration} proves quantitative polygon concentration and energy divergence for maximizing sequences.	Section~\ref{sec:global-proof} establishes the scalar conformal	correction and derives the barycenter and polynomial Gram	determinant inequalities.

	\section{Proof of the moment inequalities}\label{sec:pure-proof}
	\subsection{Preliminaries}\label{sec:analysis}
	We begin with standard entropy estimates from the Gibbs variational principle and its localization to measurable sets (see \cite{DonskerVaradhan1975,Lions1985LimitI}).
	\begin{lemma}\label{lem:entropy-tools}
		Let $\sigma$ be a probability measure on a measurable space $X$ and $\rho$ a probability density	of finite entropy.
		
		\emph{(i).} If $V$ is measurable, bounded below, and $\int_X e^V\dd\sigma<\infty$, then
		\begin{equation}\label{eq:entropy-duality}
			\int_X V\rho\dd\sigma\le\Ent(\rho)+\log\int_X e^V\dd\sigma.
		\end{equation}
		Moreover, on a measurable set $A$, with $p=\int_A\rho\dd\sigma>0$,
		\begin{equation}\label{eq:entropy-local}
			\int_A V\rho\dd\sigma
			\le\int_A\rho\log\rho\dd\sigma-p\log p
			+p\log\int_Ae^V\dd\sigma.
		\end{equation}
		When $p=0$, both sides are interpreted as zero.
		
		\emph{(ii).} If $0<\sigma(A)<1$ and
		$p=\int_A\rho\dd\sigma$, then
		\begin{equation}\label{eq:entropy-two-region}
			\begin{aligned}
				\Ent(\rho)&\ge p\log\frac{p}{\sigma(A)}
				+(1-p)\log\frac{1-p}{1-\sigma(A)}\\
				&\ge p\log\frac{1}{\sigma(A)}-\log2.
			\end{aligned}
		\end{equation}
	\end{lemma}
	\begin{proof}
		Applying Jensen's inequality to the probability measure $p^{-1}\rho\mathbf1_A\dd\sigma$, we obtain, for bounded $V$,
		\[
		\frac{1}{p}\int_A(V-\log\rho)\rho\dd\sigma
		\le \log\left(\frac{1}{p}\int_A e^V\dd\sigma\right).
		\]
		Multiplying by $p$ and rearranging gives \eqref{eq:entropy-local}. The case of general $V$ follows by	truncation. Taking $A=X$ yields \eqref{eq:entropy-duality}.
		
		Taking $V=0$ and applying \eqref{eq:entropy-local}
		separately to $A$ and $A^c$, we obtain
		\eqref{eq:entropy-two-region} from
		\[
		p\log p+(1-p)\log(1-p)\ge-\log 2.
		\]
	\end{proof}
	
	To estimate the logarithmic interaction, we introduce the kernel and its associated potential
	\begin{equation}\label{def K T}
		K(z,w)=-\log|1-z\cdot\bar w|,\qquad
		T_\rho(z)=\int_{\mathbb{S}^{2N-1}} K(z,w)\rho(w)\dd\sigma(w).
	\end{equation}

	To justify the subsequent use of logarithmic potentials and differentiation under the integral sign, we first establish some basic estimates for the singular kernel.
	\begin{lemma}\label{lem:poles}
		\emph{(i).} For $0<r<1$ and $z\in \mathbb{S}^{2N-1}$, let
		\begin{equation}\label{def Ar z}
			A_r(z)=\{w\in \mathbb{S}^{2N-1}:|1-z\cdot\bar w|<r\}.
		\end{equation}
		Then, for every $0<a<N$, we have
		\begin{equation}\label{sup Ar}
			\sigma(A_r(z))\le C_Nr^N,
		\end{equation}
		and
		\begin{equation}\label{int d^-a}
			\int_{A_r(z)}|1-z\cdot\bar w|^{-a}\dd\sigma(w)\le C_{N,a}r^{N-a}.
		\end{equation}
		In particular,
		\begin{equation}\label{int glob bd}
			A_{N,a}:=	\sup_{z\in\overline{\mathbb B^N}}\int_{\mathbb{S}^{2N-1}}|1-z\cdot\bar w|^{-a}\dd\sigma(w)\le1+\frac{C_Na}{N-a}.
		\end{equation}

		\emph{(ii).} If $\rho\ge0$, $\int_{\mathbb{S}^{2N-1}}\rho\dd\sigma=1$, and $\Ent(\rho)<\infty$, then
		\begin{equation}\label{bd T rho}
			-\log2\le T_\rho(z)\le\frac{\Ent(\rho)+\log A_{N,a}}{a},
		\end{equation}
		and
		\begin{equation}\label{bd K}
			\iint_{\mathbb{S}^{2N-1}\times\mathbb{S}^{2N-1}} |K(z,w)|\rho(z)\rho(w)\dd\sigma(z)\dd\sigma(w)<\infty.
		\end{equation}

		\emph{(iii).} 
		If $\rho\in L^\infty$, then
		\[
		-\int_{\mathbb{S}^{2N-1}}\Log(1-v\cdot\bar w)\rho(w)\dd\sigma(w),
		\qquad
		\int_{\mathbb{S}^{2N-1}}\frac{v\cdot\bar w}{1-v\cdot\bar w}
		\rho(w)\dd\sigma(w)
		\]
		extend to bounded continuous functions on	$\overline{\mathbb B^N}$, with the second being the radial derivative of the first in $\mathbb B^N$.
		
	\end{lemma}
	\begin{proof}
		By unitary invariance, we may assume that $z=e_1$. Then the first coordinate $u=w_1$ has distribution
		$$
		\frac{N-1}{\pi}(1-|u|^2)^{N-2}\dd A(u)
		$$
		on the unit disk, where $A$ denotes planar Lebesgue measure. Since $1-|u|^2\le 2r$ whenever $|1-u|<r$, and the region $|1-u|<r$ has area at most $\pi r^2$ in the unit disk, we obtain \eqref{sup Ar}.
		
		To prove \eqref{int d^-a}, we decompose $A_r(z)$ into the annuli	$2^{-j-1}r\le |1-u|<2^{-j}r$, $j\ge0$.
		Applying  \eqref{sup Ar} to each annulus, we obtain
		\begin{equation*}
			\int_{A_r(z)}|1-z\cdot\bar w|^{-a}\dd\sigma(w)
			\le C_N\sum_{j\ge0}
			(2^{-j-1}r)^{-a}(2^{-j}r)^N
			\le C_{N,a}r^{N-a},
		\end{equation*}
		where the series converges since $a<N$. 
		
		For $|u|\le1$ and $0\le s\le1$, we have $|1-u|\le |1-su|+1-s\le2|1-su|$. Thus~\eqref{sup Ar}, unitary invariance, and the trivial measure bound for $r\ge\frac{1}{2}$ give
		\begin{equation*}
			\sigma\{|1-v\cdot\bar w|<r\}\le C_Nr^N,
			\qquad |v|\le1,\quad 0<r<1.
		\end{equation*}
		Integrating the distribution function yields
		\begin{equation*}
			\begin{aligned}
				\int_{\mathbb{S}^{2N-1}}|1-v\cdot\bar w|^{-a}\dd\sigma(w)
				&\le1+a\int_0^1r^{-a-1}\sigma\{|1-v\cdot\bar w|<r\}\dd r\\
				&\le1+\frac{C_Na}{N-a},
			\end{aligned}
		\end{equation*}
		which proves~\eqref{int glob bd}.

		\medskip
		
		We now use these bounds to control the logarithmic potential. Since $K(z,w)\ge-\log2$ and $\int_{\mathbb{S}^{2N-1}}\rho\dd\sigma=1$, the lower bound in \eqref{bd T rho} follows immediately.
		For the upper bound, we apply \eqref{eq:entropy-duality} with
		$V(w)=a\min\{K(z,w),L\}$ for some large $L$. By \eqref{int glob bd}, we obtain
		\begin{equation*}
			a\int_{\mathbb{S}^{2N-1}}\min\{K(z,w),L\}\rho(w)\dd\sigma(w)
			\le \Ent(\rho)+\log A_{N,a}.
		\end{equation*} 
		Letting $L\to+\infty$ and applying, the monotone convergence theorem gives the upper bound in \eqref{bd T rho}. As a result, this ensures the absolute convergence of \eqref{bd K}.

		\medskip

		We turn to the continuity of the logarithmic potential and its radial derivative. The imaginary part of the principal logarithm is bounded, 	while its radial derivative has kernel		$\frac{z\cdot\bar w}{1-z\cdot\bar w}$.
		By the same annular estimates as in part (i), the two kernels have integrals bounded by	$C_Nr^N(1+|\log r|)$ and $C_Nr^{N-1}$ near the singularity, respectively. Since $N\geq 1$, both of them tend to zero uniformly as $r\to0$. Away from the singularity, the kernels are uniformly		continuous. Hence they depend continuously on the pole	in $L^1(\sigma)$, and integration against bounded $\rho$		preserves this continuity. Differentiation under the		integral sign in the open ball gives the stated radial		derivative.
		
	\end{proof}

	To express the logarithmic interaction in terms of holomorphic moments, we group the moments of each degree by setting
	\begin{equation}\label{def bk}
		b_k(\rho)=\sum_{|\gamma|=k}\frac{k!}{\gamma!}|m_\gamma(\rho)|^2, \qquad k\ge1.
	\end{equation}
	We then have the following expansion.
	
	\begin{lemma}	\label{lem:kernel-interaction}
		For any probability density $\rho$ with $\Ent(\rho)<\infty$, we have
		\begin{equation}\label{eq:interaction}
			S[\rho]=\int_{\mathbb{S}^{2N-1}} T_\rho\rho\dd\sigma
			=\sum_{k\ge1}\frac{b_k(\rho)}{k}\ge0.
		\end{equation}
		In particular, $T_1=0$ and $S[1]=0$.
		
		The same expansion holds for any real signed density	$\delta$ such that $|\delta|$ is dominated by a finite sum of probability densities with finite entropy:
		\begin{equation}\label{eq:signed-positive}
			S[\delta]=\sum_{k\ge1}\frac{1}{k}
			\sum_{|\gamma|=k}\frac{k!}{\gamma!}
			\left|\int_{\mathbb{S}^{2N-1}} z^\gamma\delta\dd\sigma\right|^2\ge0,
		\end{equation}
		where $S[\delta]$ is defined in the same way as \eqref{eq:interaction-def}. Consequently, for every $t\in\mathbb R$,
		\begin{equation}\label{eq:interaction-polarization}
			S[\rho+t\delta]
			=S[\rho]+2t\int_{\mathbb{S}^{2N-1}} T_\rho\delta\dd\sigma
			+t^2S[\delta].
		\end{equation}
	\end{lemma}
	
	\begin{proof}		
		The first equality of \eqref{eq:interaction} follows from Fubini's theorem. For $0<s<1$, the regularized kernel admits the uniformly convergent expansion
		\[
		K_s(z,w)=-\log|1-sz\cdot\bar w|
		=\Rea\sum_{k\ge1}\frac{s^k}{k}(z\cdot\bar w)^k.
		\]
		By the multinomial formula,
		\[
		\iint_{\mathbb{S}^{2N-1}\times\mathbb{S}^{2N-1}} K_s(z,w)\rho(z)\rho(w)
		\dd\sigma(z)\dd\sigma(w)
		=\sum_{k\ge1}\frac{s^k}{k}b_k(\rho).
		\]
		By Lemma~\ref{lem:poles}, we can apply the dominated convergence theorem to pass to the limit $s\uparrow1$. This proves \eqref{eq:interaction}. The identity $T_1=0$ follows from the vanishing holomorphic moments of the uniform density and the same Abel expansion.
			
		The argument extends to signed densities $\delta$,	since the domination hypothesis and Lemma~\ref{lem:poles} ensure absolute integrability of all the corresponding interactions.
		The polarization identity \eqref{eq:interaction-polarization} then follows from the symmetry and bilinearity of the kernel.	
		
	\end{proof}

	\subsection{Subcritical minimization}
	\label{sec:minimization}
	
	We first establish the entropy inequality \eqref{eq:entropy-main} in Theorem~\ref{thm:pure}, from which the function inequality \eqref{eq:pressure-high} follows by a weighted Cauchy--Schwarz inequality.
	For this purpose, we consider the free energy
	\begin{equation}\label{def cD}
		\mathcal D_\beta(\rho)
		=\Ent(\rho)-\beta NS[\rho],
		\qquad \beta>0.
	\end{equation}
	
	We begin with subcritical minimization and derive
	the Euler--Lagrange equation for minimizers.
	Let $\mathscr H_m$ denote the space of holomorphic
	polynomials of degrees $1,\ldots,m$ without
	constant terms.
	
	\begin{proposition}
		\label{prop:subcritical-minimization}
		For $0<\beta<q$, the functional $\mathcal D_\beta$ admits a minimizer in $\Acal_m$. For any $\beta>0$, every minimizer is bounded above and away from zero, and satisfies
		\begin{equation}\label{eq:moment-euler}
			\log\rho=2\beta NT_\rho+2\Rea\ell+c_\rho,
		\end{equation}
		where $\ell\in\mathscr H_m$ and
		\[
		c_\rho=-\log\int_{\mathbb{S}^{2N-1}}
		e^{2\beta NT_\rho+2\Rea\ell}\dd\sigma.
		\]
	\end{proposition}

	\begin{proof}
		\emph{Step 1: Entropy coercivity.}
		We first use the moment constraints to bound the mass	of small caps. Write $\dd\mu=\rho\dd\sigma$ and $c_m=\frac{m(m+2)}{3}$. We claim that, for sufficiently small $r$,
		\begin{equation}\label{eq:cap}
			\mu(A_r(z))\le b_r=(q-c_mr)^{-1}.
		\end{equation}
		To prove this, we introduce the classical Fej\'er polynomial 	(see \cite{Zygmund1959})
		\begin{equation}\label{def fejer}
			F_m(t)=1+2\sum_{j=1}^m(1-\frac{j}{q})\Rea t^j,\qquad |t|\le1.
		\end{equation}
		On the unit circle,
		\[
		F_m(t)=q^{-1}|1+t+\cdots+t^{q-1}|^2\ge0.
		\]
		By the maximum principle, $F_m\geq 0$ throughout the disk. Since the holomorphic moments	of degrees $1,\ldots,m$ vanish, we also have
		\begin{equation}\label{eq:Fejer-integral}
			\int_{\mathbb{S}^{2N-1}} F_m(z\cdot\bar w)\dd\mu(w)=1.
		\end{equation}

		For $|1-t|<r$, we have	$|t^j-1|\le jr$. Hence, using $F_m(1)=q$, we have
		\[
		F_m(t)\ge q-2r\sum_{j=1}^m j(1-\frac{j}{q})
		=q-\frac{m(m+2)}{3}r.
		\]
		Since $F_m\ge0$, we obtain
		\[
		1=\int_{\mathbb{S}^{2N-1}} F_m(z\cdot\bar w)\dd\mu(w)
		\ge(q-c_mr)\mu(A_r(z)),
		\]
		which proves \eqref{eq:cap}.
		
		We next use the cap bound to control the singular interaction through the local entropy inequality. Fix $0<a<N$, take $T$ sufficiently large, and set	$r=e^{-T}$, so by \eqref{def Ar z}, $A_r(z)=\{w:K(z,w)>T\}$.
		By \eqref{int d^-a},
		\[
		\int_{A_r(z)}e^{a(K(z,w)-T)}\dd\sigma(w)
		\le C_{N,a}e^{-NT}\le1.
		\]
		Writing $p(z)=\mu(A_r(z))$, we apply \eqref{eq:entropy-local} on $A_r(z)$ with $V(w)=a(K(z,w)-T)$ to obtain
		\begin{equation}\label{eq:cap-local-application}
			\begin{aligned}
				a\int_{A_r(z)}(K-T)\rho\dd\sigma
				&\le\int_{A_r(z)}\rho\log\rho\dd\sigma
				-p(z)\log p(z)+p(z)\log\int_{A_r(z)}
				e^{a(K-T)}\dd\sigma\\
				&\le\int_{A_r(z)}\rho\log_+\rho\dd\sigma
				-p(z)\log p(z).
			\end{aligned}
		\end{equation}
		The inequality is understood trivially when $p(z)=0$.
		
		Suppose that $p(z)\le b\le1$ for all $z$. Integrating \eqref{eq:cap-local-application} against $\mu(z)$ and using the symmetry of the cap relation, we obtain
		\begin{equation}\label{eq:localized-tail}
			\begin{aligned}
				&a\iint_{\mathbb{S}^{2N-1}\times\mathbb{S}^{2N-1}}(K(z,w)-T)_+
				\dd\mu(z)\dd\mu(w)\\
				&\quad\le\int_{\mathbb{S}^{2N-1}} p(w)\rho(w)\log_+\rho(w)
				\dd\sigma(w)-\int_{\mathbb{S}^{2N-1}} p(z)\log p(z)\dd\mu(z)\\
				&\quad\le b\bigl(\Ent(\rho)+e^{-1}\bigr)
				-\int_{\mathbb{S}^{2N-1}} p(z)\log p(z)\dd\mu(z),
			\end{aligned}
		\end{equation}
		where $\log_+x=\max\{\log x,0\}$. Here we also used $x\log x\ge-e^{-1}$. Since $-p\log p\le e^{-1}$ and $K\le T+(K-T)_+$, we conclude that
		\begin{equation}\label{eq:coercive}
			S[\rho]\le\frac{b}{a}\Ent(\rho)+C_{N,a,r}.
		\end{equation}
		
		It remains to choose the parameters. Since $\beta<q$,	we may take $a<N$ such that $qa>\beta N$, and then choose $r$ sufficiently small so that $a(q-c_mr)>\beta N$. Substituting $b=b_r$ into \eqref{eq:coercive} yields
		\begin{equation}\label{eq:coercivity}
			\mathcal D_\beta(\rho)\ge
			\left(1-\frac{\beta Nb_r}{a}\right)\Ent(\rho)-C,
		\end{equation}
		where the entropy coefficient is strictly positive.
		
		\medskip
		
		\emph{Step 2: Compactness and attainment.}
		Let $\rho_j$ be a minimizing sequence of $\mathcal D_\beta$. Then \eqref{eq:coercivity}	gives $\sup_j\Ent(\rho_j)\le H<\infty$. Passing to a subsequence, we may assume that
		$\mu_j=\rho_j\dd\sigma\rightharpoonup\mu$.
		By the lower semicontinuity of relative entropy, we have	$\mu=\rho\dd\sigma$ for some probability density $\rho$ and
		\begin{equation}
			\Ent(\rho)\le\liminf_j\Ent(\rho_j)\le H.
		\end{equation}
		
		We claim that $S[\rho_j]\to S[\rho]$.
		Indeed, \eqref{eq:entropy-two-region} and \eqref{sup Ar} give
		\[
		\sup_z\mu_j(A_{e^{-T}}(z))
		\le \frac{C_{N,H}}{T}
		\]
		for sufficiently large $T$ uniformly in $j$. Applying \eqref{eq:localized-tail} with $b=\frac{C_{N,H}}{T}$, and using
		$-p\log p\le b\log\left(\frac{1}{b}\right)$ for $0\le p\le b<e^{-1}$,
		we obtain
		\begin{equation}\label{eq:uniform-tail}
			\sup_j\iint_{\mathbb{S}^{2N-1}\times\mathbb{S}^{2N-1}}(K-T)_+\dd\mu_j(z)\dd\mu_j(w)
			\le C_{N,a,H}\frac{1+\log T}{T}
			\longrightarrow0
		\end{equation}
		as $T\to\infty$. The same estimate holds for $\mu$.
		
		For fixed $T$, the kernel $K_T=\min\{K,T\}$, with $K_T(z,z)=T$, is bounded and continuous. Weak convergence therefore gives convergence of the corresponding double integrals. Together with	\eqref{eq:uniform-tail}, this proves	$S[\rho_j]\to S[\rho]$.
		
		Since the mass and polynomial moments pass to the limit, $\rho\in\Acal_m$. Hence
		\[
		\mathcal D_\beta(\rho)
		\le\liminf_j\mathcal D_\beta(\rho_j)
		=\inf_{\Acal_m}\mathcal D_\beta,
		\]
		and the minimum is attained.
		
		\medskip

		\emph{Step 3: The Euler--Lagrange equation.}
		Let $\rho\in\Acal_m$ be a minimizer for some $\beta>0$. We first show that $\rho>0$ almost everywhere. The variation	$\rho_t=(1-t)\rho+t\in\mathcal{A}_m$. Writing $p=\sigma\{\rho=0\}$, convexity of $x\log x$ gives
		\begin{equation}\label{eq:entropy-mixing}
			\Ent(\rho_t)\le(1-t)\Ent(\rho)+pt\log t.
		\end{equation}
		Since $T_1=0$, the polarization identity	\eqref{eq:interaction-polarization} yields $S[\rho_t]=(1-t)^2S[\rho]$. Consequently,
		\[
		\mathcal D_\beta(\rho_t)-\mathcal D_\beta(\rho)
		\le pt\log t-t\Ent(\rho)
		+\beta N(2t-t^2)S[\rho].
		\]
		If $p>0$, the right-hand side is negative for sufficiently small $t>0$, contradicting minimality.	Thus $\rho>0$ almost everywhere.
		
		We may now take the admissible variations	$\rho(1+t\psi)$, where $\psi$ is real, bounded, and 
		\begin{equation*}
			\int_{\mathbb{S}^{2N-1}} z^\gamma\psi\rho\dd\sigma=0,\qquad 0\leq |\gamma|\leq m.
		\end{equation*}
		The first variation gives
		\[
		\int_{\mathbb{S}^{2N-1}}(1+\log\rho-2\beta NT_\rho)
		\psi\rho\dd\sigma=0.
		\]
		Then the standard Lagrange multiplier argument yields \eqref{eq:moment-euler} for some $\ell\in\mathscr H_m$ and $c_\rho\in\mathbb R$. Exponentiating~\eqref{eq:moment-euler} and using $\int_{\mathbb{S}^{2N-1}}\rho\dd\sigma=1$ gives the formula for $c_\rho$. The potential bound in~\eqref{bd T rho} and boundedness of $\ell$ then imply $0<c\le\rho\le C$.
	\end{proof}

	\subsection{Rigidity of the Euler--Lagrange equation}
	\label{sec:hardy-rigidity}
	
	We now establish the rigidity of the constrained Euler--Lagrange equation. The proof relies on two estimates for the same quantity: an identity derived from the equation and a strictly stronger	lower bound obtained through Hardy space methods.
	
	\begin{proposition}\label{thm:rigidity}
		If $\rho\in\Acal_m$ satisfies	\eqref{eq:moment-euler} for some $0<\beta\le q$, then $\rho=1$, $\ell=0$, and $c_\rho=0$.
	\end{proposition}

	Let $H^2$ be the $L^2(\sigma)$ closure of holomorphic polynomials, with inner product	$\langle u,v\rangle=\int_{\mathbb{S}^{2N-1}} u\bar v\dd\sigma$, linear in the first argument. For the homogeneous	decomposition $g=\sum_{d\ge0}g_d$, define
	\[
	\NN(g)=\sum_{d\ge1}d\|g_d\|_2^2,
	\qquad
	R=\sum_{i=1}^N z_i\partial_{z_i}.
	\]

	We begin with the identity derived from the Euler--Lagrange equation.

	\begin{lemma}\label{lem:radial}
		Suppose $\rho\in\Acal_m$ satisfies \eqref{eq:moment-euler} with
		$\beta>0$. Then $\rho=|g|^2$ for a normalized holomorphic function $g$ with $g,Rg\in H^2$, and
		\begin{equation}\label{eq:general-radial}
			\frac{\NN(g)}{N}=\beta\sum_{k\ge q}b_k(\rho).
		\end{equation}
	\end{lemma}
	\begin{proof}
		
		By the potential bound \eqref{bd T rho}, $T_\rho$ is bounded	above and below. Then the Euler--Lagrange equation \eqref{eq:moment-euler}	implies $0<c\le\rho\le C$ almost everywhere.
		
		For $v\in\mathbb B^N$, we define
		\[
		F(v)=-\beta N\int_{\mathbb{S}^{2N-1}}\Log(1-v\cdot\bar w)
		\rho(w)\dd\sigma(w)+\ell(v)
		\]
		and set $g(v)=e^{F(v)+c_\rho/2}$. By Lemma~\ref{lem:poles}(iii), both $F$ and $RF$ extend continuously and boundedly to
		$\overline{\mathbb B^N}$, with
		\[
		RF(v)=\beta N\int_{\mathbb{S}^{2N-1}}
		\frac{v\cdot\bar w}{1-v\cdot\bar w}
		\rho(w)\dd\sigma(w)+R\ell(v).
		\]
		Taking real parts on the sphere and using \eqref{eq:moment-euler}, we obtain $2\Rea F=\log\rho-c_\rho$. Hence $|g|^2=\rho$ and $\|g\|_2=1$. Moreover, $g$ and $Rg=(RF)g$ are bounded holomorphic functions and therefore belong to $H^2$.
		
		For $0<s<1$, expanding the  kernel into its uniformly convergent geometric series	and integrating term by term gives
		\begin{align*}
			&\int_{\mathbb{S}^{2N-1}} RF(sz)\rho(z)\dd\sigma(z)\\
			&=\beta N\sum_{k\ge1}s^k
			\iint_{\mathbb{S}^{2N-1}\times\mathbb{S}^{2N-1}}(z\cdot\bar w)^k\rho(z)\rho(w)
			\dd\sigma(z)\dd\sigma(w)\\
			&\quad+\int_{\mathbb{S}^{2N-1}} R\ell(sz)\rho(z)\dd\sigma(z)\\
			&=\beta N\sum_{k\ge1}s^kb_k(\rho)
			+\int_{\mathbb{S}^{2N-1}} R\ell(sz)\rho(z)\dd\sigma(z),
		\end{align*}
		where $b_k$ is defined in \eqref{def bk}.	By the moment constraints, the last integral vanishes and $b_k(\rho)=0$ for $1\le k\le m$. Since $RF$ extends continuously to the closed ball,	we may let $s\uparrow1$ to obtain
		\[
		\int_{\mathbb{S}^{2N-1}} RF\rho\dd\sigma
		=\beta N\sum_{k\ge q}b_k(\rho).
		\]
		On the other hand, the homogeneous expansions of $g$ and $Rg$ converge in $H^2$. By orthogonality,
		\[
		\NN(g)=\langle Rg,g\rangle
		=\int_{\mathbb{S}^{2N-1}}(Rg)\bar g\dd\sigma
		=\int_{\mathbb{S}^{2N-1}} RF\,|g|^2\dd\sigma
		=\int_{\mathbb{S}^{2N-1}} RF\rho\dd\sigma.
		\]
		This proves \eqref{eq:general-radial}.
	\end{proof}

	The radial identity obtained above relates $\NN(g)$ to the moments of $\rho$. To derive a lower bound	from the moment constraints, we now use the	multiplication operators on $H^2$ and their adjoints.
	
	Let $S_i h=z_i h$ on $H^2$. For multi-indices $\gamma$, we write $S^\gamma=S_1^{\gamma_1}\cdots S_N^{\gamma_N}$	and let $S^{*\gamma}=(S^\gamma)^*$ denote the adjoint operator. The monomials form an orthogonal basis of $H^2$	with the standard norm formula	(see \cite{Rudin1980})
	\begin{equation}\label{eq:monomial-norm}
		\langle z^\eta,z^\nu\rangle=0\quad(\eta\ne\nu),	\qquad
		\|z^\eta\|_2^2
		=\frac{\eta!(N-1)!}{(N+|\eta|-1)!}
		=\frac{\eta!}{|\eta|!D_{|\eta|}}.
	\end{equation}
	Consequently, for $|\eta|=d$ and $\gamma\le\eta$, we have
	\begin{equation}\label{eq:adjoint}
		S^{*\gamma}z^\eta
		=\frac{\eta!}{(\eta-\gamma)!}\frac{(N+d-|\gamma|-1)!}{(N+d-1)!}z^{\eta-\gamma},	
	\end{equation}
	and $S^{*\gamma}z^\eta=0$ otherwise.
	
	\begin{lemma}\label{lem:budget}
		For $g=\sum_{d\ge0}g_d\in H^2$,
		\begin{equation}\label{eq:adjoint-sum}
			\frac{\NN(g)}{N}
			=\sum_{k\ge1}\sum_{|\gamma|=k}\frac{k!}{\gamma!}
			\|S^{*\gamma}g\|_2^2,
		\end{equation}
		where both sides may be infinite.
	\end{lemma}

	\begin{proof}
		Fix a homogeneous degree $d\ge1$ and	$1\le k\le d$. By the adjoint formula \eqref{eq:adjoint}, the operator
		\[
		\sum_{|\gamma|=k}\frac{k!}{\gamma!}
		S^\gamma S^{*\gamma}
		\]
		acts on a monomial $z^\eta$, $|\eta|=d$,	by multiplication with
		\[
		\frac{(N+d-k-1)!}{(N+d-1)!}
		\sum_{\substack{|\gamma|=k\\\gamma\le\eta}}
		\frac{k!}{\gamma!}
		\frac{\eta!}{(\eta-\gamma)!}.
		\]
		Comparing the coefficients of	$t^k$ in $	\prod_{i=1}^N(1+t)^{\eta_i}=(1+t)^d$ gives
		\[
		\sum_{\substack{|\gamma|=k\\\gamma\le\eta}}
		\frac{\eta!}{\gamma!(\eta-\gamma)!}
		=\binom{d}{k}.
		\]
		Thus the scalar is
		\[
		\frac{(N+d-k-1)!}{(N+d-1)!}
		\frac{d!}{(d-k)!}
		=\frac{D_{d-k}}{D_d},
		\]
		independently of $\eta$. Consequently, for every homogeneous polynomial $g_d$ of degree $d$, we have
		\[
		\sum_{|\gamma|=k}\frac{k!}{\gamma!}
		\|S^{*\gamma}g_d\|_2^2
		=\frac{D_{d-k}}{D_d}\|g_d\|_2^2.
		\]
		
		Summing over $k=1,\ldots,d$ and using
		\[
		\sum_{j=0}^{d-1}\binom{N+j-1}{N-1}
		=\binom{N+d-1}{N},
		\]
		we obtain
		\[
		\sum_{k=1}^d\frac{D_{d-k}}{D_d}
		=\frac{\binom{N+d-1}{N}}{D_d}
		=\frac{d}{N}.
		\]
		Finally, for each fixed $\gamma$, the images $S^{*\gamma}g_d$ of distinct homogeneous degrees are orthogonal. Summing over $d$ and applying Tonelli's theorem to the nonnegative terms proves \eqref{eq:adjoint-sum}.
	\end{proof}

	The moment constraints do not imply orthogonality of the vectors $z^\eta g$, since their inner products	involve mixed moments. To control their Gram matrix, we introduce a coefficient norm on homogeneous	polynomials. For
	\[
	p_j(z)=\sum_{|\eta|=j}
	p_\eta\sqrt{\frac{j!}{\eta!}}\,z^\eta,
	\qquad j\ge0,
	\]
	define
	\[
	\|p_j\|_{\mathrm{cf}}^2
	\coloneqq\sum_{|\eta|=j}|p_\eta|^2
	=D_j\|p_j\|_2^2.
	\]
	For $\rho\in\Acal_m$, let $\mathsf C_m(\rho)$ be the Gram matrix in $L^2(\rho\dd\sigma)$ of the	monomials $\sqrt{\frac{j!}{\eta!}}\,z^\eta$, where $1\le j=|\eta|\le m$, that is,
	\begin{equation*}
		(\mathsf C_m(\rho))_{\eta,\nu}=\langle \sqrt{\frac{|\eta|!}{\eta!}}\,z^\eta, \sqrt{\frac{|\nu|!}{\nu!}}\,z^\nu \rangle_{L^2(\rho\dd\sigma)}.
	\end{equation*}
	Set	$L_m(\rho)=\lambda_{\max}(\mathsf C_m(\rho))$. The following estimate provides the required control of the radial energy.
	
	\begin{lemma}\label{lem:projection}
		Let $g\in H^2$ satisfy $\|g\|_2=1$ and 	$\rho=|g|^2\in\Acal_m$. Then	$0<L_m(\rho)<1$ and
		\begin{equation}\label{eq:joint-budget}
			\frac{\NN(g)}{N}\ge\left(1+\frac{m}{L_m(\rho)}\right)\sum_{k\ge q}b_k(\rho).
		\end{equation}
	\end{lemma}
	
	\begin{proof}
		We first prove that	$L_m(\rho)<1$. Set
		\[
		h\coloneqq\lambda_{\max}\left(\int_{\mathbb{S}^{2N-1}} zz^*\rho\dd\sigma\right).
		\]
		Since the matrix is positive semidefinite with trace one, $h\le1$. In fact, $h<1$. Otherwise, there would be a unit vector $v$ such that $\int_{\mathbb{S}^{2N-1}}|v^*z|^2\rho\dd\sigma=1$. Since $|v^*z|\le1$ on the sphere, $\rho\dd\sigma$ would be supported on the complex line through $v$, a set of zero surface measure for $N\ge2$. This contradicts absolute continuity with respect to $\sigma$.

		We claim that for every polynomial	$p=\sum_{j=1}^d p_j$, $1\le d\le m$,
		\begin{equation}\label{est p^2}
			\int_{\mathbb{S}^{2N-1}}|p|^2\rho\dd\sigma
			\le\sum_{j=1}^{d-1}\|p_j\|_{\mathrm{cf}}^2+h\|p_d\|_{\mathrm{cf}}^2.
		\end{equation}
		We prove this by induction on $d$. The case $d=1$ follows from the definition of $h$.
		
		For the induction step, we use the identities
		\begin{equation}\label{cf id}
			p_j=\sum_{i=1}^N z_i\frac{\partial_{z_i}p_j}{j},
			\qquad
			\sum_{i=1}^N\left\|\frac{\partial_{z_i}p_j}{j}\right\|_{\mathrm{cf}}^2=\|p_j\|_{\mathrm{cf}}^2.
		\end{equation}
		The second identity follows by the definition of $\|\cdot \|_{\mathrm{cf}}$ and the fact that each normalized coefficient $p_\eta$ contributes $\sum_i(\frac{\eta_i}{j})|p_\eta|^2=|p_\eta|^2$.
		
		Writing
		\[
		p=\sum_{i=1}^N z_i
		\left(\partial_{z_i}p_1+
		\sum_{j=2}^d\frac{\partial_{z_i}p_j}{j}\right).
		\]
		By the Cauchy--Schwarz inequality and $|z|=1$,
		\[
		\int_{\mathbb{S}^{2N-1}}|p|^2\rho\dd\sigma
		\le\|p_1\|_{\mathrm{cf}}^2
		+\sum_{i=1}^N
		\int_{\mathbb{S}^{2N-1}}\left|
		\sum_{j=2}^d\frac{\partial_{z_i}p_j}{j}
		\right|^2\rho\dd\sigma,
		\]
		where the cross terms vanish by the moment constraints. Applying the induction hypothesis to each polynomial
		in the last sum gives
		\[
		\begin{aligned}
			\int_{\mathbb{S}^{2N-1}}|p|^2\rho\dd\sigma
			&\le\|p_1\|_{\mathrm{cf}}^2
			+\sum_{j=2}^{d-1}\sum_{i=1}^N
			\left\|\frac{\partial_{z_i}p_j}{j}
			\right\|_{\mathrm{cf}}^2+h\sum_{i=1}^N
			\left\|\frac{\partial_{z_i}p_d}{d}
			\right\|_{\mathrm{cf}}^2\\
			&=\sum_{j=1}^{d-1}\|p_j\|_{\mathrm{cf}}^2
			+h\|p_d\|_{\mathrm{cf}}^2,
		\end{aligned}
		\]
		where the last equality follows from \eqref{cf id}. This completes the induction.

		For any nonzero $p$, applying~\eqref{est p^2} at its highest nonzero homogeneous degree gives
		\[
		\int_{\mathbb{S}^{2N-1}}|p|^2\rho\dd\sigma
		<\sum_{j=1}^m\|p_j\|_{\mathrm{cf}}^2.
		\]
		Therefore, compactness of its unit sphere yields
		\begin{equation}\label{eq:gram-strict}
			0<L_m(\rho)<1.
		\end{equation}
		
		\medskip
		
		We now turn to the Hardy space projection.	Let $V_m$ be the span of	$\sqrt{\frac{j!}{\eta!}}\,z^\eta g$, where $1\le j=|\eta|\le m$. The moment constraints give $g\perp V_m$. By the Gram matrix bound, for every $v\in H^2$,
		\[
		\sum_{j=1}^m\sum_{|\eta|=j}
		\frac{j!}{\eta!}
		|\langle v,z^\eta g\rangle|^2
		\le L_m(\rho)\|P_{V_m}v\|_2^2,
		\]
		where $P_{V_m}$ denotes orthogonal projection onto $V_m$. Since $g\perp V_m$ and $\|g\|_2=1$, we obtain
		\[
		\|v\|_2^2\ge|\langle v,g\rangle|^2
		+\frac{1}{L_m(\rho)}
		\sum_{j=1}^m\sum_{|\eta|=j}
		\frac{j!}{\eta!}
		|\langle v,z^\eta g\rangle|^2.
		\]
		
		Apply this inequality to $v=S^{*\gamma}g$. By the adjoint relation,	$\langle S^{*\gamma}g,z^\eta g\rangle	=\overline{m_{\gamma+\eta}(\rho)}$. Multiplying by $\frac{k!}{\gamma!}$ and summing over $|\gamma|=k$, the terms of degree $j$ contribute
		\[
		\sum_{|\gamma|=k}\sum_{|\eta|=j}
		\frac{k!j!}{\gamma!\eta!}
		|m_{\gamma+\eta}(\rho)|^2
		=b_{k+j}(\rho),
		\]
		where we used the multinomial identity
		\[
		\sum_{\substack{\gamma+\eta=\delta\\
				|\gamma|=k,\ |\eta|=j}}
		\frac{k!j!}{\gamma!\eta!}
		=\frac{(k+j)!}{\delta!}.
		\]
		Consequently,
		\[
		\sum_{|\gamma|=k}\frac{k!}{\gamma!}
		\|S^{*\gamma}g\|_2^2
		\ge b_k(\rho)
		+\frac{1}{L_m(\rho)}
		\sum_{j=1}^m b_{k+j}(\rho).
		\]
		Summing over $k\ge1$ and using	Lemma~\ref{lem:budget}, we obtain
		\[
		\frac{\NN(g)}{N}
		\ge\sum_{k\ge1}b_k(\rho)
		+\frac{1}{L_m(\rho)}
		\sum_{j=1}^m\sum_{k\ge1}b_{k+j}(\rho).
		\]
		Since $b_k(\rho)=0$ for $1\le k\le m$, each sum $\sum_{k\ge1}b_{k+j}(\rho)$, $1\le j\le m$, equals $\sum_{n\ge q}b_n(\rho)$. This proves~\eqref{eq:joint-budget}. 
	\end{proof}

	\begin{proof}[Proof of Proposition~\ref{thm:rigidity}]
		Lemma~\ref{lem:radial} gives a normalized $g$ with $|g|^2=\rho$ and
		\begin{equation}\label{eq:radial-euler}
			\frac{\NN(g)}{N}=\beta\sum_{k\ge q}b_k(\rho).
		\end{equation}
		Together with \eqref{eq:joint-budget}, this gives
		\[
		\beta\sum_{k\ge q}b_k
		\ge\left(1+\frac{m}{L_m(\rho)}\right)\sum_{k\ge q}b_k(\rho).
		\]
		By~\eqref{eq:gram-strict}, $1+\frac{m}{L_m(\rho)}>q\ge\beta$. The moment sum $\sum_{k\ge q}b_k(\rho)$ is finite by~\eqref{eq:radial-euler}, and hence must be zero. The same identity gives $\NN(g)=0$, so $g$ is constant and $\rho=1$.	Since $T_1=0$, averaging~\eqref{eq:moment-euler} gives $c_\rho=0$, and comparison of homogeneous terms then yields $\ell=0$.
	\end{proof}

	\subsection{Proof of the main inequalities}
	\label{sec:duality}
	
	We now complete the proof of Theorem~\ref{thm:pure}.
	\begin{proof}[Proof of Theorem~\ref{thm:pure}]
		We start with the identity
		\begin{equation}\label{expre Pf}
			P(f)=2\Rea\int_{\mathbb{S}^{2N-1}}(f-f_0)\rho_f\dd\sigma-\Ent(\rho_f).
		\end{equation}
		We estimate the two terms separately.
		
		For $0<\beta<q$, Proposition~\ref{prop:subcritical-minimization} gives a minimizer of $\mathcal D_\beta$, which must be the uniform density by Proposition~\ref{thm:rigidity}. Hence $\mathcal D_\beta\ge0$ on $\Acal_m$. Letting $\beta\uparrow q$ yields \eqref{eq:entropy-main}.	If equality holds, $\rho$ is a minimizer at $\beta=q$, and the Euler--Lagrange equation together with rigidity implies $\rho=1$.

		For the first term on the right-hand side of \eqref{expre Pf}, we write $f-f_0=\sum_{\gamma\ne0}f_\gamma z^\gamma$ and set $s_\gamma=\frac{N(|\gamma|-1)!}{\gamma!}$. By \eqref{eq:monomial-norm}, we have
		\[
		E_N(f)=\sum_{\gamma\ne0}\frac{|f_\gamma|^2}{s_\gamma},
		\qquad
		NS[\rho_f]=\sum_{\gamma\ne0}
		s_\gamma|m_\gamma(\rho_f)|^2.
		\]
		Since the moments up to degree $m$ vanish,	the weighted Cauchy--Schwarz inequality gives
		\begin{equation}\label{eq:high-pairing}
			\left|\int_{\mathbb{S}^{2N-1}}(f-f_0)\rho_f\dd\sigma\right|^2\le NE_{>m}(f)S[\rho_f].
		\end{equation}
		
		Using \eqref{expre Pf} and \eqref{eq:high-pairing} and completing the square, we obtain
		\begin{equation}\label{eq:function-deficit}
			\begin{aligned}
				&\frac{1}{q}E_{>m}(f)-P(f)\\
				\ge&\Ent(\rho_f)-qNS[\rho_f]+q\left(\sqrt{NS[\rho_f]}-\frac{\sqrt{E_{>m}(f)}}{q}\right)^2\ge0.
			\end{aligned}
		\end{equation}
		This proves \eqref{eq:pressure-high}. Equality implies $\rho_f=1$, which forces $f$	to be constant. Conversely, constants attain equality. Taking $m=1$ gives Corollary~\ref{thm:centered}.

		\medskip	
		
		It remains to show the optimality of both coefficients in	Theorem~\ref{thm:pure}.

		Let $\omega=e^{2\pi i/q}$ be a primitive $q$-th root of unity. For $0<r<1$, consider	$f_r(z)=-N\Log(1-rz_1^q)$. Then the corresponding density $\rho_{f_r}$ is invariant under $z\mapsto\omega z$.  Thus $m_\gamma(\rho_{f_r})=0$ for	$1\le|\gamma|\le m$, and hence $\rho_{f_r}\in\Acal_m$. Since $f_r$ contains only homogeneous terms of degrees divisible by $q$, we have	$E_{>m}(f_r)=E_N(f_r)$.

		By the power series expansion and monomial orthogonality, we obtain
		\begin{align}
			E_N(f_r)
			&=qN\log\frac{1}{1-r^2},
			\label{eq:sharp-energy}\\
			Z_{f_r}
			&=\sum_{j\ge0}\frac{D_j^2}{D_{qj}}r^{2j}
			\sim q^{1-N}(1-r^2)^{-N}.
			\label{eq:sharp-Z}
		\end{align}
		Here the asymptotic follows from $\frac{D_j}{D_{qj}}\to q^{1-N}$ and $\sum_{j\ge0}D_jr^{2j}=(1-r^2)^{-N}$.
		Consequently, we have
		\[
		\frac{P(f_r)}{E_N(f_r)}\longrightarrow\frac{1}{q},
		\qquad
		\frac{1}{q}E_N(f_r)-P(f_r)
		\longrightarrow(N-1)\log q.
		\]
		This shows that $\frac{1}{q}$ is optimal in \eqref{eq:pressure-high}.

		For the entropy inequality \eqref{eq:entropy-main}, consider the normalized
		Hardy kernels
		\[
		B_{r,v}(z)
		=\frac{(1-r^2)^N}{|1-rz\cdot\bar v|^{2N}},
		\qquad v\in \mathbb{S}^{2N-1}.
		\]
		By the Hardy reproducing formula,	$\int_{\mathbb{S}^{2N-1}} hB_{r,v}\dd\sigma=h(rv)$ for holomorphic $h$ on a neighborhood of the closed ball. Taking $h=1$ and
		$h(z)=\Log(1-rz\cdot\bar v)$ gives
		\[
		\int_{\mathbb{S}^{2N-1}} B_{r,v}\dd\sigma=1,
		\qquad
		\Ent(B_{r,v})=N\log\frac{1}{1-r^2},
		\]
		respectively. We then define the probability density
		\[
		\rho_r=\frac{1}{q}\sum_{j=0}^{q-1}
		B_{r,\omega^je_1}.
		\]
		By the reproducing formula,
		\[
		m_\gamma(\rho_r)
		=\frac{1}{q}\sum_{j=0}^{q-1}
		(r\omega^je_1)^\gamma.
		\]
		Using the root-of-unity identity, we obtain $m_\gamma(\rho_r)=0$ for $1\le|\gamma|\le m$, and $b_k(\rho_r)=r^{2k}$ if $q\mid k$,	and zero otherwise. Thus $\rho_r\in\Acal_m$, and
		\[
		S[\rho_r]=\frac{1}{q}\log\frac{1}{1-r^{2q}},
		\qquad
		\Ent(\rho_r)\le N\log\frac{1}{1-r^2},
		\]
		where the entropy bound follows from convexity. Consequently,
		\[
		qN\le\frac{\Ent(\rho_r)}{S[\rho_r]}
		\le qN\frac{\log\left(\frac{1}{1-r^2}\right)}{\log\left(\frac{1}{1-r^{2q}}\right)}
		\longrightarrow qN.
		\]
		Hence the coefficient $qN$ is optimal in \eqref{eq:entropy-main}. This completes the proof of Theorem~\ref{thm:pure}.
	\end{proof}
	
	\section{Quantitative concentration}
	\label{sec:quantitative-concentration}
	
	In this section, we establish the quantitative concentration results. We begin with the following lemma, which shows that near-saturation of the cap mass bound implies closeness to a uniform regular $q$-polygon measure.
	
	\begin{lemma}\label{lem:cap-rigidity}
		Let $\dd\mu=\rho\dd\sigma$ with $\rho\in\Acal_m$. Suppose that, for $v\in\mathbb{S}^{2N-1}$, $r>0$, and $\varepsilon\ge0$,
		\begin{equation}\label{mu Br v}
			\mu\{z:|z-v|<r\}\ge\frac{1}{q}-\varepsilon.
		\end{equation}
		Then there exists a constant $C_q$ depending only on $q$ such that
		\begin{equation}\label{eq:cap-rigidity}
			W_1(\mu,\pi_v)^2\le C_q(\varepsilon+r^2).
		\end{equation}
	\end{lemma}
	
	\begin{proof}
		We recall the Fej\'er polynomial defined in \eqref{def fejer}. For $|w|=1$ and $w\ne1$,
		\begin{equation*}
			F_m(w)=\frac{|1-w^q|^2}{q|1-w|^2}
			=\frac{2(1-\Rea w^q)}{q|1-w|^2}
			\ge\frac{1-\Rea w^q}{2q}.
		\end{equation*}
		The inequality also holds at $w=1$. By the maximum principle, 
		\begin{equation}\label{eq:Fejer-harmonic-lower}
			F_m(w)\ge\frac{1-\Rea w^q}{2q},\qquad |w|\le1.
		\end{equation}
		
		For $|w|\leq 1$, we write $w$ in the form $w=a\omega^s e^{i\theta}$, where $0\le a\le1$, $s\in\mathbb{Z}$, and $|\theta|\le\frac{\pi}{q}$. Then we always have
		\begin{equation*}
			a^q\cos(q\theta)\le a\cos\theta.
		\end{equation*}
		
		For $w=z\cdot\overline v$,
		\begin{equation*}
			\dist(z,\Pi_v)^2
			=2\left(1-\max_{0\le\ell<q}\Rea(\omega^{-\ell}w)\right)
			=2(1-a\cos\theta).
		\end{equation*}
		Combining this with \eqref{eq:Fejer-harmonic-lower} gives
		\begin{equation}\label{eq:Fejer-distance}
			F_m(z\cdot\overline v)\ge\frac{\dist(z,\Pi_v)^2}{4q}.
		\end{equation}

		For $|w|=1$, we have the elementary inequality
		\begin{equation*}
			1-\Rea w^k\le k^2(1-\Rea w),
		\end{equation*}
		so by the maximum principle, the same inequality holds for all $|w|\leq 1$. Consequently,
		\begin{equation}\label{q-Fq up}
			q-F_m(w)=2\sum_{k=1}^{q-1}\left(1-\frac{k}{q}\right)(1-\Rea w^k)\le2A_q(1-\Rea w),
		\end{equation}
		where
		\begin{equation*}
			A_q=\sum_{k=1}^{q-1}\left(1-\frac{k}{q}\right)k^2=\frac{q(q^2-1)}{12}.
		\end{equation*}
		Taking $w=z\cdot\overline v$ in~\eqref{q-Fq up} yields
		\begin{equation}\label{eq:Fejer-cap}
			F_m(z\cdot\overline v)\ge q-A_q|z-v|^2.
		\end{equation}

		Let $\mathcal{B}_{r}(v)=\{z: |z-v|<r\}$. Then by \eqref{mu Br v}, $p:=\mu(\mathcal{B}_{r}(v))\ge\frac{1}{q}-\varepsilon$. Equations
		\eqref{eq:Fejer-integral} and \eqref{eq:Fejer-cap} imply
		\begin{equation}\label{int F out}
			\int_{\mathbb{S}^{2N-1}\setminus \mathcal{B}_{r}(v)}F_m(z\cdot\overline v)\dd\mu(z)\le1-qp+A_qr^2p\le q\varepsilon+A_qr^2.
		\end{equation}
		On $\mathcal{B}_{r}(v)$, $\dist(z,\Pi_v)\le|z-v|<r$. Using
		\eqref{eq:Fejer-distance} and \eqref{int F out} on $\mathbb{S}^{2N-1}\setminus \mathcal{B}_{r}(v)$, we obtain
		\begin{equation}\label{eq:D-bound}
			D^2:=\int_{\mathbb{S}^{2N-1}}\dist(z,\Pi_v)^2\dd\mu(z)
			\le4q^2\varepsilon+(1+4qA_q)r^2.
		\end{equation}
		This shows that most of the measure lies close to $\Pi_v$.

		For each $z\in\mathbb{S}^{2N-1}$, let $T(z)$ be a nearest vertex $\omega^\ell v\in\Pi_v$. If the nearest vertex is not unique, choose the one with the smallest index $\ell\in\{0,1,\ldots,q-1\}$. Then the measure $\nu:=T_\#\mu$ is a discrete measure on $\Pi_v$. Write
		\begin{equation*}
			\nu=\sum_{s=0}^{q-1}a_s\delta_{\omega^sv},
			\qquad a_s\ge0,\qquad \sum_{s=0}^{q-1}a_s=1.
		\end{equation*}
		The coupling $(z,T(z))_\#\mu$ and the Cauchy--Schwarz inequality give
		\begin{equation}\label{eq:nearest-transport}
			W_1(\mu,\nu)\le\int_{\mathbb{S}^{2N-1}}|z-T(z)|\dd\mu(z)\le D.
		\end{equation}
		For $|x|,|y|\le1$, $|x^k-y^k|\le k|x-y|$. Hence, for $1\le k<q$, we have
		\begin{align*}
			\left|\sum_{s=0}^{q-1}a_s\omega^{sk}\right|
			&=\left|\int_{\mathbb{S}^{2N-1}}
			\left[(T(z)\cdot\overline v)^k-(z\cdot\overline v)^k\right]
			\dd\mu(z)\right|\\
			&\le k\int_{\mathbb{S}^{2N-1}}|T(z)-z|\dd\mu(z)\le kD.
		\end{align*}
		The identity $\sum_{k=0}^{q-1}\omega^{k(s-t)}=q\,\mathbf1_{\{s=t\}}$ then gives
		\begin{equation*}
			\sum_{s=0}^{q-1}|a_s-\frac{1}{q}|^2=\frac{1}{q}\sum_{k=1}^{q-1}\left|\sum_{s=0}^{q-1}a_s\omega^{sk}\right|^2
			\le\frac{D^2}{q}\sum_{k=1}^{q-1}k^2.
		\end{equation*}
		Therefore, the Cauchy--Schwarz inequality gives
		\begin{equation*}
			\sum_{s=0}^{q-1}|a_s-\frac{1}{q}|\le D\left(\sum_{k=1}^{q-1}k^2\right)^{1/2}.
		\end{equation*}
		
		On the other hand, keep the common mass at each vertex fixed and transport the remaining	mass between vertices. The remaining mass is	$\frac{1}{2}\sum_s|a_s-\frac{1}{q}|$ and every transport distance is at most two. Thus
		\begin{equation*}
			W_1(\nu,\pi_v)\le\sum_{s=0}^{q-1}|a_s-\frac{1}{q}|.
		\end{equation*}
		Combining this with \eqref{eq:nearest-transport} and \eqref{eq:D-bound}
		proves \eqref{eq:cap-rigidity}. For example, one may take
		\begin{equation*}
			C_q=
			\left(1+\sqrt{\sum_{k=1}^{q-1}k^2}\right)^2
			\max\{4q^2,\,1+4qA_q\}.
		\end{equation*}
	\end{proof}

	\begin{proof}[Proof of Theorem \ref{thm:quantitative}]
		
		For a probability density $\rho$, define
		\begin{equation*}
			m_\delta(z)=\int_{A_\delta(z)}\rho(w)\dd\sigma(w),
			\qquad M_\delta=\sup_{z\in\mathbb{S}^{2N-1}}m_\delta(z).
		\end{equation*}
		For $0<a<N$ and $0<\delta<1$, apply~\eqref{eq:entropy-local} on $A_\delta(z)$ with $V(w)=aK(z,w)$ and integrate against $\rho(z)\dd\sigma(z)$. Using the symmetry of the cap relation, $\int\rho\log_+\rho\dd\sigma\le\Ent(\rho)+e^{-1}$, and $-p\log p\le e^{-1}$, we obtain
		\[
		a\iint_{A_{\delta}(z)}K(z,w)\rho(z)\rho(w)\dd\sigma(z)\dd\sigma(w)
		\le M_\delta\Ent(\rho)+\log^+A_{N,a}+\frac{2}{e}.
		\]
		On the complementary region, $K(z,w)\le\log(1/\delta)$. Hence
		\begin{equation}\label{eq:localized}
			S[\rho]\le\frac{M_\delta}{a}\Ent(\rho)+\frac{\log^+A_{N,a}}{a}+\frac{2}{ae}+\log\frac{1}{\delta}.
		\end{equation}
		
		We first prove \eqref{eq:density-rate}. For simplicity, write $H=\Ent(\rho)$, $S=S[\rho]$. If $0\le H<2$, the fraction in the right-hand side of \eqref{eq:density-rate} is at least $\frac{\log 2}{3}$. Since $d_q^2\le4$, \eqref{eq:density-rate} holds for   $C_{N,q}\geq \frac{12}{\log 2}$. 
		
		In what follows, we assume $H\ge2$. Define
		\begin{equation*}
			\theta=\frac{H-qNS}{H}\in (0,1],\qquad
			a=N-\frac{1}{H},\qquad
			\delta=\frac{1}{H}.
		\end{equation*}
		Our choice satisfies	$N-\frac{1}{2}\le a<N$, so \eqref{int glob bd} gives	$A_{N,a}\le C_NH$. By \eqref{eq:localized}, we obtain
		\begin{align*}
			M_\delta
			&\ge \frac{a S}{H}
			-C_N\frac{\log H}{H}\\
			&=\left(1-\frac{1}{NH}\right)\frac{1-\theta}{q}
			-C_N\frac{\log H}{H}\\
			&\ge\frac{1}{q}-\frac{\theta}{q}
			-C_{N,q}\frac{\log H}{H}.
		\end{align*}
		By the definition of $m_\delta$, there exists $v\in\mathbb{S}^{2N-1}$ such that
		\begin{equation*}
			m_\delta(v)\ge M_\delta-\frac{1}{H}
			\ge\frac{1}{q}-\frac{\theta}{q}       -C_{N,q}\frac{\log H}{H}.
		\end{equation*}
		
		By \eqref{eq:chord-cap},
		\begin{equation*}
			A_\delta(v)\subset
			\{w\in\mathbb{S}^{2N-1}:|w-v|<\sqrt{\frac{2}{H}}\}.
		\end{equation*}
		Applying Lemma~\ref{lem:cap-rigidity} with
		\begin{equation*}
			r^2=\frac{2}{H},\qquad
			\varepsilon=\frac{\theta}{q}+C_{N,q}\frac{\log H}{H},
		\end{equation*}
		we obtain
		\begin{equation*}
			d_q(\rho\dd\sigma)^2
			\le C_{N,q}\left(\theta+\frac{\log H}{H}\right).
		\end{equation*}
		This proves \eqref{eq:density-rate}.
		
		\medskip
		
		Now we turn to \eqref{eq:source-rate}. We define the deficit
		\begin{equation*}
			\Delta=E_N(f)-qP(f).
		\end{equation*}
		As in \eqref{eq:function-deficit}, weighted Cauchy--Schwarz gives
		\begin{equation}\label{eq:deficit-comparison}
			\Delta\ge q\bigl(\Ent(\rho_f)-qNS[\rho_f]\bigr)
			+\left(\sqrt{E_N(f)}-q\sqrt{NS[\rho_f]}\right)^2.
		\end{equation}
		Write $E=E_N(f)$. If $E=0$, then $f=0$ and \eqref{eq:source-rate} holds after increasing $C_{N,q}$ if necessary, since $d_q\le2$. We may therefore assume $E>0$ and put $\eta=\frac{\Delta}{E}$. Jensen's inequality and \eqref{eq:pressure-high} imply $0\le\eta\le1$. 
		
		If $E\le1$, then
		\begin{equation*}
			\frac{\Delta+\log(2+E)}{1+E}\ge\frac{\log2}{2}.
		\end{equation*}
		If $E>1$ and $\eta>\frac{1}{4}$, then
		\begin{equation*}
			\frac{\Delta+\log(2+E)}{1+E}
			\ge\frac{E}{4(1+E)}\ge\frac{1}{8},
		\end{equation*}
		In both cases, $d_q^2\le4$ gives \eqref{eq:source-rate}.

		Now we assume  $\eta\le\frac{1}{4}$. Write $H=\Ent(\rho_f)$ and $S=S[\rho_f]$. By \eqref{eq:deficit-comparison}, we have
		\begin{equation*}
			(1-\sqrt\eta)^2E
			\le q^2NS\le(1+\sqrt\eta)^2E,
			\qquad
			0\le H-qNS\le\frac{\eta E}{q}.
		\end{equation*}
		Therefore, 
		\begin{equation}\label{eq:H-comparison}
			\frac{E}{4q}\le H
			\le\frac{(1+\sqrt\eta)^2+\eta}{q}E\le\frac{5E}{2q}.
		\end{equation}
		Applying \eqref{eq:density-rate} and using \eqref{eq:deficit-comparison}--\eqref{eq:H-comparison}, we obtain
		\begin{align*}
			d_q(\rho_f\dd\sigma)^2
			&\le C_{N,q}\frac{H-qNS+\log(2+H)}{1+H}\\
			&\le C_{N,q}
			\frac{\frac{\Delta}{q}+\log(2+\frac{5E}{2q})}{1+\frac{E}{4q}}\\
			&\le C_{N,q}\frac{\Delta+\log(2+E)}{1+E}.
		\end{align*}
		This completes the proof of Theorem~\ref{thm:quantitative}.
	\end{proof}
	
	\medskip
	
	The rest of this section is devoted to the proof of Theorem \ref{thm:maximizing}. For $k\ge1$, put $\lambda_k=\frac{kD_k}{N}$, where $D_k$ is defined in~\eqref{def Dk lak}. We first show that in the small-energy regime, the ratio $\frac{P(f)}{E_N(f)}$ is controlled by $\frac{1}{\lambda_q}$.
	
	\begin{lemma}\label{lem:local}
		There is an absolute constant $C$ such that every admissible $f$ with $f(0)=0$ and $0<E_N(f)\le1$ satisfies
		\begin{equation}\label{eq:local}
			P(f)\le\frac{E_N(f)}{\lambda_q}+C E_N(f)^{3/2}.
		\end{equation}
	\end{lemma}
	
	\begin{proof}
		Put $E=E_N(f)$, $u=f+\overline f$, and $X=\frac{u}{\sqrt E}$. Applying \eqref{eq:ordinary} to $\frac{f}{\sqrt E}$, we  obtain
		\begin{equation*}
			\int_{\mathbb{S}^{2N-1}}e^{tX}\dd\sigma\le e^{t^2},
			\qquad t\in\mathbb R.
		\end{equation*}
		Consequently, for $s\ge0$, we have
		\begin{equation}\label{eq:uniform-exponential}
			\int_{\mathbb{S}^{2N-1}}e^{s|X|}\dd\sigma\le2e^{s^2}.
		\end{equation}
		
		Define the remainder $R=e^u-1-u$. Then $|R|\le\frac{1}{2}|u|^2e^{|u|}$. Since $E\le1$, using \eqref{eq:uniform-exponential}, we obtain
		\begin{equation*}
			\|R\|_2^2
			\le\frac{E^2}{4}\int_{\mathbb{S}^{2N-1}}|X|^4e^{2|X|}\dd\sigma
			\le C E^2.
		\end{equation*}

		For $1\le|\gamma|<q$, let
		\begin{equation*}
			e_\gamma(z)=
			\left(\frac{(N+|\gamma|-1)!}{(N-1)!\gamma!}\right)^{1/2}z^\gamma.
		\end{equation*}
		These are orthonormal by \eqref{eq:monomial-norm}. By orthogonality,
		\begin{equation*}
			0=\int_{\mathbb{S}^{2N-1}}\overline{e_\gamma}e^u\dd\sigma
			=\int_{\mathbb{S}^{2N-1}}\overline{e_\gamma}f\dd\sigma
			+\int_{\mathbb{S}^{2N-1}}\overline{e_\gamma}R\dd\sigma.
		\end{equation*}
		Here we used
		$\int_{\mathbb{S}^{2N-1}}\overline{e_\gamma}\dd\sigma=\int_{\mathbb{S}^{2N-1}}\overline{e_\gamma}\,\overline f\dd\sigma=0$. Consequently,
		\begin{equation}\label{eq:low-modes}
			\sum_{k=1}^{q-1}\|f_k\|_2^2
			=\sum_{1\le|\gamma|<q}|\langle f,e_\gamma\rangle_2|^2
			\le\|R\|_2^2\le CE^2.
		\end{equation}
		Since $\lambda_k$ increases with $k$, we obtain
		\begin{equation}\label{eq:L2-local}
			\|f\|_2^2=\sum_{k<q}\|f_k\|_2^2+\sum_{k\ge q}\|f_k\|_2^2
			\le CE^2+\frac{E}{\lambda_q}.
		\end{equation}
		
		By orthogonality, we have
		\begin{equation*}
			\int_{\mathbb{S}^{2N-1}}u\dd\sigma=0,\qquad
			\int_{\mathbb{S}^{2N-1}}f^2\dd\sigma=0,\qquad
			\int_{\mathbb{S}^{2N-1}}u^2\dd\sigma=2\|f\|_2^2.
		\end{equation*}
		Therefore, Taylor expansion and \eqref{eq:uniform-exponential} yield
		\begin{equation}\label{est e^u high}
			\left|\int_{\mathbb{S}^{2N-1}}e^u\dd\sigma-1-\|f\|_2^2\right|
			\le\frac{1}{6}\int_{\mathbb{S}^{2N-1}}|u|^3e^{|u|}\dd\sigma
			\le\frac{E^{3/2}}{6}\int_{\mathbb{S}^{2N-1}}|X|^3e^{|X|}\dd\sigma
			\le CE^{3/2}.
		\end{equation}
		Since $P(f)\le \int_{\mathbb{S}^{2N-1}}e^u\dd\sigma-1$, \eqref{eq:local} follows from \eqref{eq:L2-local} and \eqref{est e^u high}.
		
	\end{proof}

	We observe that $D_1=N$	and $D_k$ increases strictly in $k$ for $N\ge2$.  Since $q\ge2$, we have
	\begin{equation*}
		\lambda_q=\frac{qD_q}{N}>q.
	\end{equation*}
	Therefore, the coefficient $\frac{1}{\lambda_q}$ in \eqref{eq:local} is strictly less than $\frac{1}{q}$, the sharp coefficient in \eqref{eq:pressure-high}. This suggests that a maximizing sequence cannot have a uniformly bounded energy.
	
	\begin{lemma}\label{lem EN f infty}
		Under the assumptions of Theorem~\ref{thm:maximizing}, we have
		\begin{equation*}
			E_N(f_j)\to\infty.
		\end{equation*}
	\end{lemma}
	
	\begin{proof}
		Let $E_j=E_N(f_j)>0$. Suppose on the contrary that $E_j\to M_0\in[0,\infty)$ after passing to a subsequence. Then there exists $f$ with finite energy such that
		\begin{equation*}
			f_j\rightharpoonup f\ \text{in the energy space}.
		\end{equation*}
		The lower semi-continuity gives
		\begin{equation*}
			E_N(f)\le\liminf_{j\to\infty}E_N(f_j)\leq M_0.
		\end{equation*}
		Moreover, we have the strong convergence $f_j\longrightarrow f\ \text{in }L^2(\sigma)$.
		
		By the elementary inequality $|e^a-e^b|\le |a-b|(e^a+e^b)$ and Cauchy Schwarz inequality, we obtain
		\begin{equation*}
			\|e^{2\Rea f_j}-e^{2\Rea f}\|_{L^1}\le 	2\|f_j-f\|_{L^2}(\|e^{2\Rea f_j}\|_{L^2}+\|e^{2\Rea f}\|_{L^2})\longrightarrow 0.
		\end{equation*}
	 The continuity of the logarithm then gives $P(f_j)\longrightarrow P(f)$. As a consequence, we also have
		\begin{equation*}
			\rho_{f_j}\longrightarrow\rho_f\ \text{in }L^1(\sigma).
		\end{equation*} 
		Therefore, we also have $\rho_f\in\Acal_m$.	Since $\{f_j\}$ is maximizing, we have $$P(f)=\lim_{j\to\infty} \left(\frac{P(f_j)}{E_j}E_j\right)=\frac{M_0}{q}.$$

		If $M_0>0$, then $f$ attains the equality in the sharp inequality \eqref{eq:pressure-high}, so $f=0$. This contradicts $P(f)=\frac{M_0}{q}>0$.
		
		If $M_0=0$, then by Lemma~\ref{lem:local}, for all sufficiently large $j$, we have
		\begin{equation*}
			\frac{P(f_j)}{E_j}
			\le\frac{1}{\lambda_q}+C\sqrt{E_j}.
		\end{equation*}
		Taking the upper limit yields
		\begin{equation*}
			\frac{1}{q}\le\frac{1}{\lambda_q}<\frac{1}{q},
		\end{equation*}
		which is also a contradiction. Thus we proved that $E_j\to\infty$.
	\end{proof}
	
	\begin{proof}[Proof of Theorem~\ref{thm:maximizing}]
		By Lemma~\ref{lem EN f infty}, we have proved that $E_j=E_N(f_j)\to\infty$. Define the deficit
		$$\Delta_j=E_j-qP(f_j).$$
		Then the maximizing property says that		$\frac{\Delta_j}{E_j}\to0$. Thus, by \eqref{eq:source-rate}, we obtain
		\begin{equation*}
			d_q(\rho_{f_j}\dd\sigma)^2
			\le C_{N,q}\left[\frac{E_j}{1+E_j}	\left(1-q\frac{P(f_j)}{E_j}\right)
			+\frac{\log(2+E_j)}{1+E_j}\right]\longrightarrow0.
		\end{equation*}

		Now for each $j$, we can choose $v_j\in\mathbb{S}^{2N-1}$ so that
		\begin{equation*}
			W_1(\rho_{f_j}\dd\sigma,\pi_{v_j})
			\le d_q(\rho_{f_j}\dd\sigma)+\frac{1}{j}.
		\end{equation*}
		Suppose that a subsequence of $\rho_{f_j}\dd\sigma$ converges weakly
		to $\mu$. Let $v_j\to v$ after passing to a further subsequence. Pairing corresponding vertices shows that
		\begin{equation*}
			W_1(\pi_{v_j},\pi_v)
			\le\frac{1}{q}\sum_{s=0}^{q-1}
			|\omega^sv_j-\omega^sv|=|v_j-v|\longrightarrow0.
		\end{equation*}
		The triangle inequality now gives
		$W_1(\rho_{f_j}\dd\sigma,\pi_v)\to0$, so $\mu=\pi_v$. This identifies every subsequential weak limit.
		
		Finally, \eqref{eq:deficit-comparison} gives
		\begin{equation*}
			\left|1-q\sqrt{\frac{NS[\rho_{f_j}]}{E_j}}\right|
			\le\sqrt{\frac{\Delta_j}{E_j}}\longrightarrow0.
		\end{equation*}
		Hence $\frac{NS[\rho_{f_j}]}{E_j}\to\frac{1}{q^2}$. Also,
		\begin{equation*}
			0\le
			\frac{\Ent(\rho_{f_j})-qNS[\rho_{f_j}]}{E_j}
			\le\frac{\Delta_j}{qE_j}\longrightarrow0,
		\end{equation*}
		which proves the first limit in \eqref{eq:entropy-ratios}. Thus we complete the proof of Theorem \ref{thm:maximizing}.
	\end{proof}

	
	\section{Conformal transformations and global corrections}
	\label{sec:global-proof}
	
	We now extend the centered inequality to arbitrary functions by means of conformal transformations. The main idea is to choose a transformation that    centers the associated density and then apply   Theorem~\ref{thm:pure}. This yields the scalar conformal correction in Theorem~\ref{thm:global}, from which the barycenter and Gram determinant inequalities follow.
	
	\subsection{The weighted conformal action}
	\label{sec:weighted}
	
	We first introduce the Lorentzian frames used	to define the conformal correction.	Let	$H=\operatorname{diag}(1,-I_N)$ and	$X(z)=(1,z_1,\ldots,z_N)^T$. A frame is a row $t=(t_0,t')\in\mathbb C^{N+1}$ satisfying
	\begin{equation*}
		tHt^*=|t_0|^2-|t'|^2=1.
	\end{equation*}
	We identify frames differing by a common phase and choose $t_0>0$ whenever a representative is needed.
	
	For $s>0$, define
	\begin{equation}\label{eq:complex-power-def}
		\begin{aligned}
			F_s(t;f)
			&=\int_{\mathbb{S}^{2N-1}}|tX|^{2s}e^{f+\bar f}\dd\sigma,
			\qquad
			\Lambda_s(f)=\inf_{tHt^*=1}F_s(t;f),\\
			A_s(f)
			&=\frac{N}{N+s}E_N(f)+2\Rea f_0.
		\end{aligned}
	\end{equation}
	Equivalently, a phase-normalized frame is determined by a point
	$a\in\mathbb B^N$ through
	\begin{equation*}
		t=\frac{(1,-a^*)}{\sqrt{1-|a|^2}}.
	\end{equation*}
	Thus the correction can also be written in the ball coordinates as
	\begin{equation*}
		\Lambda_s(f)=\inf_{a\in\mathbb B^N}		\frac{\displaystyle\int_{\mathbb{S}^{2N-1}}|1-z\cdot\bar a|^{2s}e^{f+\bar f}\dd\sigma}
		{(1-|a|^2)^s}.
	\end{equation*}
	We retain frames in our proof because matrix multiplication makes	the transformation laws transparent.
	
	We next describe the conformal transformations	associated with the Hermitian form $H$.	Let
	\begin{equation*}
		U(1,N)=\{L\in GL_{N+1}(\mathbb C):L^*HL=H\}.
	\end{equation*}
	Write a matrix $L\in U(1,N)$ in block form as
	\begin{equation*}
		L=\begin{pmatrix}\ell&b\\ u&B\end{pmatrix},
		\qquad
		a_L(z)=\ell+bz,
		\qquad
		\phi_L(z)=\frac{u+Bz}{a_L(z)},
	\end{equation*}
	where $B$ is an $N\times N$	matrix. Here $bz$ is ordinary matrix multiplication, with no conjugation. These are the standard fractional linear ball automorphisms; see \cite[Chapter~2]{Rudin1980}.
	Then $LX(z)=a_L(z)X(\phi_L(z))$.	Multiplying $L$ by a common phase leaves	$\phi_L$ unchanged, so we choose $a_L(0)>0$.
	
	Since $LHL^*=H$, we have	$|\ell|^2-|b|^2=1$,	and hence $a_L$ has no zeros on the closed ball.	Preservation of the Hermitian form gives
	\begin{equation}\label{eq:ball-kernel-transform}
		1-\phi_L(z)\cdot\overline{\phi_L(w)}=\frac{1-z\cdot\bar w}{a_L(z)\overline{a_L(w)}}.
	\end{equation}
	In particular, $\phi_L$ is an automorphism of	the unit ball, with inverse $\phi_{L^{-1}}$.
	
	The corresponding boundary Jacobian is
	\begin{equation}\label{eq:cr-jacobian}
		J_{\phi_L}=|a_L|^{-2N},
		\qquad
		\dd\sigma(\phi_Lz)=J_{\phi_L}(z)\dd\sigma(z).
	\end{equation}
	Indeed, for	$\theta=-\operatorname{Im}(X^*H\,\mathrm{d}X)$, the homogeneous-coordinate identity gives	$\phi_L^*\theta=|a_L|^{-2}\theta$,	and the formula follows from the surface form	$\theta\wedge(\mathrm{d}\theta)^{N-1}$.
	
	The weighted conformal action is simply
	\begin{equation}\label{eq:cr-weighted-action}
		T_L^{(s)}f=f\circ\phi_L-(N+s)\Log a_L,
	\end{equation}
	where the holomorphic logarithm is real at zero. The coefficient $N+s$ combines the boundary Jacobian exponent $N$ with the extra weight $s$ in $F_s$.
	
	\begin{lemma}\label{lem:cr-weighted-energy}
		For $s>0$ and $L\in U(1,N)$, the weighted conformal action $T_L^{(s)}$ preserves the finite-energy space and
		\begin{equation}\label{eq:cr-A-invariance}
			A_s(T_L^{(s)}f)=A_s(f).
		\end{equation}
		Moreover, writing $C=L^{-1}$, we have
		\begin{equation}\label{eq:cr-F-covariance}
			F_s(t;T_L^{(s)}f)=F_s(tC;f),
			\qquad
			\Lambda_s(T_L^{(s)}f)=\Lambda_s(f).
		\end{equation}
	\end{lemma}
	
	\begin{proof}
		Consider the finite-energy Hilbert space modulo constants, with inner product
		\begin{equation*}
			\langle f,h\rangle_E
			=\sum_{k\ge1}\frac{kD_k}{N}\langle f_k,h_k\rangle,
			\qquad E_N(f)=\langle f,f\rangle_E.
		\end{equation*}
		Its reproducing kernel is $\mathcal K(z,w)=-N\Log(1-z\cdot\bar w)$,	as follows from \eqref{eq:monomial-norm}. Writing $K_w=\mathcal K(\cdot,w)$, we have $\langle f,K_w\rangle_E=f(w)-f_0$.
		
		Let $a=\phi_L(0)$. By \eqref{eq:ball-kernel-transform},	with the logarithmic branches chosen consistently,
		\begin{equation*}
			\begin{aligned}
				\mathcal K(\phi_L z,\phi_L w)
				-\mathcal K(\phi_L z,a)-\mathcal K(a,\phi_L w)
				+\mathcal K(a,a)
				=\mathcal K(z,w).
			\end{aligned}
		\end{equation*}
		Thus $K_w\mapsto K_{\phi_Lw}-K_a$ extends to an isometry $U_L$ of the energy space. Its range is dense, since a function orthogonal to the range is constant. Hence $U_L$ is unitary. By the reproducing property,
		\begin{equation*}
			\langle f,U_LK_w\rangle_E
			=f(\phi_L(w))-f(a),
		\end{equation*}
		and therefore $U_L^*f=f\circ\phi_L-f(a)$.	In particular,
		\begin{equation*}
			E_N(f\circ\phi_L)=E_N(f).
		\end{equation*}
		
		We next consider the logarithmic correction. Write $a_0=a_L(0)>0$ and $c=\phi_L^{-1}(0)$.	Using $L^{-1}=HL^*H$, we obtain
		\begin{equation*}
			c=-\frac{b^*}{a_0},\qquad
			1-|c|^2=a_0^{-2},\qquad
			a_L(z)=a_0(1-z\cdot\bar c).
		\end{equation*}
		Consequently, for $v=T_L^{(s)}f$,
		\begin{equation*}
			v=f\circ\phi_L+\frac{N+s}{N}K_c-(N+s)\log a_0.
		\end{equation*}
		Since $\phi_L(c)=0$, the reproducing property gives
		\begin{equation*}
			E_N(K_c)=\mathcal K(c,c)=2N\log a_0,
			\qquad
			\langle f\circ\phi_L,K_c\rangle_E=f_0-f(a).
		\end{equation*}
		Expanding the energy therefore yields
		\begin{equation*}
			\begin{aligned}
				E_N(v)
				&=E_N(f)+\frac{2(N+s)}{N}\Rea(f_0-f(a))
				+\frac{2(N+s)^2}{N}\log a_0,\\
				v_0&=f(a)-(N+s)\log a_0.
			\end{aligned}
		\end{equation*}
		The correction terms cancel in $\frac{N}{N+s}E_N(v)+2\Rea v_0$, proving	\eqref{eq:cr-A-invariance}.
		
		For the scalar covariance, set $C=L^{-1}$	and change variables $y=\phi_L(z)$. The homogeneous-coordinate identity and
		\eqref{eq:cr-jacobian} give
		\begin{equation*}
			CX(y)=a_L(z)^{-1}X(z),\qquad
			\dd\sigma(z)=|a_L(z)|^{2N}\dd\sigma(y).
		\end{equation*}
		Together with $e^{v+\bar v}=e^{f(y)+\bar f(y)}|a_L(z)|^{-2(N+s)}$,
		we obtain
		\begin{equation*}
			\begin{aligned}
				F_s(t;v)
				=\int_{\mathbb{S}^{2N-1}}|tCX(y)|^{2s}
				e^{f(y)+\bar f(y)}\dd\sigma(y)=F_s(tC;f).
			\end{aligned}
		\end{equation*}
		Since $t\mapsto tC$ is a bijection on frames, taking the infimum gives
		$\Lambda_s(v)=\Lambda_s(f)$.
	\end{proof}
	
	We now show that the weighted conformal action can be used to center any finite-energy density. The required transformation is determined by	the minimizing frame of $\Lambda_s(f)$.
	
	\begin{lemma}\label{lem:complex-frame}
		For every $s>0$ and finite-energy $f$, the infimum defining $\Lambda_s(f)$ is positive and attained at a unique frame modulo phase.
		If $M_f=0$, the minimizing frame is $(1,0)$	and $\Lambda_s(f)=Z_f$.
		
		For general $f$, let $C\in U(1,N)$ have the minimizing frame as its first row, and set	$L=C^{-1}$ and $v_s=T_L^{(s)}f$. Then
		\begin{equation}\label{eq:complex-frame-output}
			M_{v_s}=0,\qquad
			Z_{v_s}=\Lambda_s(f),\qquad
			A_s(v_s)=A_s(f).
		\end{equation}
	\end{lemma}
	
	\begin{proof}
		We first prove existence. On the compact set
		\begin{equation*}
			\{t:|t|_{\rm Euc}=1,\ t_0\in\mathbb R,\
			t_0\ge|t'|\},
		\end{equation*}
		the function $F_s(t;f)$ is continuous and strictly positive, since $tX$ has at most one zero on the sphere and $e^{f+\bar f}>0$ almost everywhere.
		By homogeneity,	$F_s(t;f)\ge c_{f,s}|t|_{\rm Euc}^{2s}$	on frames. Hence the infimum is positive and attained.
		
		To prove uniqueness, take two distinct frame classes $t,u$. Completing $t$ to an	$H$-orthonormal basis gives $|tHu^*|>1$. After changing the phase of $u$, we may assume	$c=tHu^*>1$. Set
		\begin{equation*}
			v=\frac{u-ct}{\sqrt{c^2-1}},
			\qquad
			t(\tau)=\cosh\tau\,t+\sinh\tau\,v.
		\end{equation*}
		Then $tHv^*=0$, $vHv^*=-1$, and $t(\tau)$ joins the two frames at $\tau=\operatorname{arcosh}c$.
		
		Write $a=t(\tau)X$, $b=\dot t(\tau)X$, and $p=2s$. Since $X^*HX=0$, the Hermitian	orthogonality relations give $|b|\le|a|$. Moreover, $\ddot a=a$, so
		\begin{equation}\label{eq:complex-convexity}
			\frac{\mathrm{d}^2}{\mathrm{d}\tau^2}|a|^p
			=p|a|^{p-2}
			\left(
			|a|^2+|b|^2
			+(p-2)\frac{(\Rea(\bar a b))^2}{|a|^2}
			\right)>0.
		\end{equation}
		Indeed, the bracket is bounded below by $\min\{1,p\}|a|^2$. Since $a$ does not vanish along the path, integration gives strict convexity of	$F_s(t(\tau);f)$. Thus the minimizing frame is unique, and every stationary frame is minimizing.
		
		For a row $w\in\mathbb C^N$, consider at $(1,0)$ the variation
		\begin{equation*}
			t(\varepsilon)
			=(\sqrt{1+\varepsilon^2|w|^2},\varepsilon w).
		\end{equation*}
		Differentiating gives
		\begin{equation*}
			\left.\frac{\mathrm{d}}{\mathrm{d}\varepsilon}
			F_s(t(\varepsilon);f)\right|_{\varepsilon=0}
			=2s\Rea(w\cdot M_f).
		\end{equation*}
		Hence $(1,0)$ minimizes $F_s$ precisely when	$M_f=0$.
		
		Finally, complete the minimizing frame $t$ to an $H$-orthonormal basis, and let	$C\in U(1,N)$ be the resulting matrix with first row $t$. Set $L=C^{-1}$ and	$v_s=T_L^{(s)}f$.
		By \eqref{eq:cr-F-covariance},		$F_s(u;v_s)=F_s(uC;f)$, so $(1,0)$ is	the minimizing frame for $v_s$. The preceding variation and	Lemma~\ref{lem:cr-weighted-energy} now give
		\begin{equation*}
			M_{v_s}=0,\qquad
			Z_{v_s}=\Lambda_s(f),\qquad
			A_s(v_s)=A_s(f).
		\end{equation*}
	\end{proof}
	
	\subsection{The scalar correction}
	
	We now apply the centered inequality to the function $v_s$ obtained in Lemma~\ref{lem:complex-frame}. By conformal invariance and the centering property,
	\begin{equation*}
		\begin{aligned}
			A_s(f)-\log\Lambda_s(f)
			&=\frac{N}{N+s}E_N(v_s)+2\Rea(v_s)_0-\log Z_{v_s}\\
			&=\frac{N}{N+s}E_N(v_s)-P(v_s).
		\end{aligned}
	\end{equation*}
	Since $M_{v_s}=0$, Theorem~\ref{thm:pure} gives	$\frac{1}{2}E_N(v_s)-P(v_s)\ge0$. Consequently,
	\begin{equation}\label{eq:complex-deficit}
		A_s(f)-\log\Lambda_s(f)
		=\left(\frac{1}{2}E_N(v_s)-P(v_s)\right)
		+\frac{N-s}{2(N+s)}E_N(v_s).
	\end{equation}
	
	\begin{proof}[Proof of Theorem~\ref{thm:global}]
		For $0<s\le N$, both terms in \eqref{eq:complex-deficit} are nonnegative	by the centered inequality. Equality therefore	forces $v_s$ to be constant, by	Corollary~\ref{thm:centered}.
		
		To characterize the equality cases, we use the inverse weighted action. For $C=L^{-1}$,	the identity $a_L(\phi_Cz)a_C(z)=1$ gives
		\begin{equation*}
			\Log a_L\circ\phi_C+\Log a_C=0
		\end{equation*}
		with the normalized logarithm branches. Hence $T_C^{(s)}T_L^{(s)}f=f$, and the inverse image of a constant has the form $c-(N+s)\Log a_C$. Writing $a_C(z)=a_C(0)(1-z\cdot\bar a)$, where $a\in\mathbb B^N$, we recover	\eqref{eq:weighted-extremals}.
		Conversely, every function in this family is transformed to a constant and attains equality.
		
		For optimality of the coefficient $\frac{N}{N+s}$, we take $f_r=-(N+s)\Log(1-rz_1)$, $0<r<1$.
		These are equality functions, and we have
		\begin{equation*}
			E_N(f_r)=\frac{(N+s)^2}{N}
			\log\frac{1}{1-r^2}\longrightarrow\infty.
		\end{equation*}
		Hence the coefficient $\frac{N}{N+s}$ cannot be	decreased.
		
		For $s>N$, consider instead the centered sequence $h_r=-N\Log(1-rz_1^2)$.By  Lemma~\ref{lem:complex-frame},	$\Lambda_s(h_r)=Z_{h_r}$, and
		\begin{equation*}
			\log\Lambda_s(h_r)-A_s(h_r)
			=P(h_r)-\frac{N}{N+s}E_N(h_r)\longrightarrow+\infty,
		\end{equation*}
		since $\frac{P(h_r)}{E_N(h_r)}\to\frac{1}{2}>\frac{N}{N+s}$ and $E_N(h_r)\to\infty$.
	\end{proof}
	
	\subsection{The barycenter correction}
	
	\begin{proof}[Proof of Corollary~\ref{cor:barycenter}]
		Since $|z|=1$ and $\rho_f>0$ almost everywhere,	we have $|M_f|<Z_f$. Set
		$B_f=(Z_f^2-|M_f|^2)^{1/2}$.
		
		For a phase-normalized frame $t=(t_0,t')$, the inequalities $|tX|\ge\Rea(tX)$ and $t_0=\sqrt{1+|t'|^2}$ give
		\begin{equation*}
			\begin{aligned}
				F_{\frac{1}{2}}(t;f)
				\ge t_0Z_f+\Rea(t'\cdot M_f)\ge Z_f\sqrt{1+|t'|^2}-|M_f||t'|\ge B_f.
			\end{aligned}
		\end{equation*}
		Taking the infimum over frames yields
		\begin{equation}\label{eq:bary-power-compare}
			B_f\le\Lambda_{\frac{1}{2}}(f).
		\end{equation}
		The scalar inequality at $s=\frac{1}{2}$ proves	\eqref{eq:barycenter-main}.
		
		For equality, \eqref{eq:complex-deficit} forces	the centered representative to be constant. Moreover, equality in \eqref{eq:bary-power-compare}	requires $t'=0$, since otherwise	$|tX|>\Rea(tX)$ on a set of positive measure. The corresponding conformal transformation is	then unitary, so $f$ is constant.
		Conversely, constants attain equality.
		
		To prove optimality, take the scalar extremals
		\begin{equation*}
			f_r=-\left(N+\frac{1}{2}\right)\Log(1-rz_1),
			\qquad 0<r<1.
		\end{equation*}
		By Theorem~\ref{thm:global},
		\begin{equation*}
			\log\Lambda_{\frac{1}{2}}(f_r)
			=\frac{2N}{2N+1}E_N(f_r),
			\qquad
			E_N(f_r)=\frac{(N+\frac{1}{2})^2}{N}
			\log\frac{1}{1-r^2}.
		\end{equation*}
		
		We claim that
		\begin{equation*}
			B_{f_r}\ge c_N\Lambda_{\frac{1}{2}}(f_r)
		\end{equation*}
		for some $c_N>0$ independent of $r$ near one. Indeed, writing $\kappa=N+\frac{1}{2}$, we have
		\begin{equation*}
			Z_{f_r}
			=\int_{\mathbb{S}^{2N-1}}|1-rz_1|^{-2\kappa}\dd\sigma
			\asymp(1-r^2)^{-N-1},
		\end{equation*}
		and
		\begin{equation*}
			1-|\frac{M_{f_r}}{Z_{f_r}}|^2
			\asymp1-r^2.
		\end{equation*}
		To justify the second estimate, rotational symmetry	gives $M_{f_r}=(M_r,0,\ldots,0)$, where $M_r$ is real. Moreover,
		\begin{equation*}
			1-\frac{M_r}{Z_{f_r}}
			=\frac{1}{Z_{f_r}}
			\int_{\mathbb{S}^{2N-1}}(1-\Rea z_1)|1-rz_1|^{-2\kappa}
			\dd\sigma
			\asymp1-r^2.
		\end{equation*}
		The last estimate, together with the estimate for $Z_{f_r}$, follows from one-variable boundary scaling. Since $\frac{M_r}{Z_{f_r}}\to1$, this yields	the asserted barycenter estimate. Consequently,
		\begin{equation*}
			B_{f_r}
			=Z_{f_r}\sqrt{1-|\frac{M_{f_r}}{Z_{f_r}}|^2}
			\asymp(1-r^2)^{-\kappa}.
		\end{equation*}
		Since	$\Lambda_{\frac{1}{2}}(f_r)=(1-r^2)^{-\kappa}$, the claim follows.
		
		To summarize, we obtain
		\begin{equation*}
			\log B_{f_r}
			=\frac{2N}{2N+1}E_N(f_r)+O(1).
		\end{equation*}
		As $E_N(f_r)\to\infty$, the energy coefficient	$\frac{2N}{2N+1}$ cannot be decreased.
	\end{proof}
	
	\subsection{Polynomial Gram determinants}
	
	Since the basis polynomials are bounded and	$e^{f+\bar f}>0$ almost everywhere, the Gram matrix	$G_r(f)$ is positive definite, and its determinant is independent of the orthonormal basis.
	
	\begin{proof}[Proof of Corollary~\ref{cor:hol-gram}]
		We first establish the conformal covariance of the Gram matrix. To use the linear action of $U(1,N)$, homogenize the basis polynomials by
		\begin{equation*}
			\widetilde p_i(Y)
			=Y_0^r p_i(\frac{Y_1}{Y_0},\ldots,\frac{Y_N}{Y_0}).
		\end{equation*}
		For $C\in GL_{N+1}(\mathbb C)$, let $R_r(C)$ be the induced substitution matrix on the	degree-$r$ homogeneous polynomial space. For diagonal $C$, the monomial basis gives
		\begin{equation*}
			\det R_r(C)
			=(\det C)^{\binom{N+r}{N+1}},
		\end{equation*}
		since each diagonal entry occurs with total	exponent $\frac{rd_r}{N+1}$.	By similarity invariance, the identity holds for all diagonalizable matrices, and hence for all	invertible matrices by density.
		In particular, $|\det R_r(C)|=1$	for $C\in U(1,N)$.
		
		For $L=C^{-1}$, homogeneity and the conformal change of variables give
		\begin{equation}\label{eq:hol-gram-covariance}
			G_r(T_L^{(r)}f)
			=R_r(C)G_r(f)R_r(C)^*,
			\qquad
			\Psi_r(T_L^{(r)}f)=\Psi_r(f).
		\end{equation}
		Indeed, the factor $|a_L|^{2r}$ from the homogeneous polynomials cancels the corresponding weight in the transformed measure.
		
		Now we choose the centered representative $v_r$ from Lemma~\ref{lem:complex-frame}.	The normalization of the basis and the	multinomial formula give
		\begin{equation*}
			\sum_{i=1}^{d_r}|p_i(z)|^2
			=\sum_{k=0}^r D_k=d_r.
		\end{equation*}
		Consequently,	$\operatorname{tr}G_r(v_r)=d_rZ_{v_r}$.	By the arithmetic--geometric mean inequality,
		\begin{equation*}
			\det G_r(v_r)\le Z_{v_r}^{d_r},
		\end{equation*}
		with equality precisely when		$G_r(v_r)=Z_{v_r}I_{d_r}$.  Thus
		\begin{equation}\label{eq:hol-gram-comparison}
			\Psi_r(f)\le\log Z_{v_r}
			=\log\Lambda_r(f).
		\end{equation}
		
		For $1\le r\le N$, Theorem~\ref{thm:global} now gives
		\begin{equation*}
			\Psi_r(f)\le\log\Lambda_r(f)\le A_r(f),
		\end{equation*}
		which proves \eqref{eq:hol-gram-main}.		Equality forces $v_r$ to be constant, and hence $f$ belongs to the family		\eqref{eq:weighted-extremals} with $s=r$. Conversely, these functions attain equality. The energy coefficient is optimal by the same equality sequence as in	Theorem~\ref{thm:global}.
	\end{proof}

\section*{Acknowledgments}
The authors used ChatGPT to explore proof strategies and assist with language editing. They take responsibility for the
mathematical statements, proofs, references, and conclusions
of this paper.


\begin{thebibliography}{99}
		
		\bibitem{Aubin1979}
		T. Aubin, \emph{Meilleures constantes dans le th\'eor\`eme d'inclusion de Sobolev et un th\'eor\`eme de Fredholm non lin\'eaire pour la transformation conforme de la courbure scalaire}, J. Funct. Anal. \textbf{32} (1979), 148--174.
		
		\bibitem{BartolucciMalchiodi2013}
		D. Bartolucci and A. Malchiodi,
		\href{https://doi.org/10.1007/s00220-013-1731-0}
		{\emph{An improved geometric inequality via vanishing moments,
				with applications to singular Liouville equations}},
		Comm. Math. Phys. \textbf{322} (2013), 415--452.
		
		\bibitem{Beckner1993}
		W. Beckner,
		\href{https://doi.org/10.2307/2946638}
		{\emph{Sharp Sobolev inequalities on the sphere and the
				Moser--Trudinger inequality}},
		Ann. of Math. (2) \textbf{138} (1993), 213--242.
		
		\bibitem{Berman2009}
		R. J. Berman,
		\href{https://arxiv.org/abs/0905.4263}
		{\emph{Sharp inequalities for determinants of Toeplitz
				operators and $\bar\partial$-Laplacians on line bundles}},
		arXiv:0905.4263 (2009).
		
		\bibitem{BFM2013}
		T. P. Branson, L. Fontana and C. Morpurgo,
		\href{https://doi.org/10.4007/annals.2013.177.1.1}
		{\emph{Moser--Trudinger and Beckner--Onofri's
				inequalities on the CR sphere}},
		Ann. of Math. (2) \textbf{177} (2013), 1--52.
		
		\bibitem{Burbea1982}
		J. Burbea,
		\href{https://doi.org/10.2996/kmj/1138036562}
		{\emph{Norm inequalities of exponential type for
				holomorphic functions}},
		Kodai Math. J. \textbf{5} (1982), 339--354.
		
		\bibitem{Burbea1983}
		J. Burbea,
		\href{https://doi.org/10.1090/S0002-9947-1983-0684506-0}
		{\emph{Inequalities for holomorphic functions of
				several complex variables}},
		Trans. Amer. Math. Soc. \textbf{276} (1983), 247--266.
		
		%
		
		\bibitem{ChangGui2023}
		S.-Y. A. Chang and C. Gui,
		\href{https://doi.org/10.1002/cpa.22034}
		{\emph{A sharp inequality on the exponentiation
				of functions on the sphere}},
		Comm. Pure Appl. Math. \textbf{76} (2023), 1303--1326.
		
		\bibitem{ChangHang2022}
		S.-Y. A. Chang and F. Hang,
		\href{https://doi.org/10.1002/cpa.21952}
		{\emph{Improved Moser--Trudinger--Onofri inequality
				under constraints}},
		Comm. Pure Appl. Math. \textbf{75} (2022), 197--220.
		
		\bibitem{ChangYang1987}
		S.-Y. A. Chang and P. C. Yang,
		\href{https://doi.org/10.1007/BF02392560}
		{\emph{Prescribing Gaussian curvature on $S^2$}},
		Acta Math. \textbf{159} (1987), 215--259.
		
		\bibitem{ChangYang1988}
		S.-Y. A. Chang and P. C. Yang,
		\href{https://doi.org/10.4310/jdg/1214441783}
		{\emph{Conformal deformation of metrics on $S^2$}},
		J. Differential Geom. \textbf{27} (1988), 259--296.

        \bibitem{ChenLu2026}
L. Chen and G. Lu,
\emph{Sharp Moser--Onofri inequality on the CR sphere under
zero-moment constraints and its application to the stability
of the CR Moser--Onofri inequality},
preprint,
\href{https://arxiv.org/abs/2610.29877}{arXiv:2610.29877}
(2026).

		\bibitem{ChenWeiWu2024}
		X. Chen, W. Wei and N. Wu,
		\emph{Almost sharp Sobolev trace inequalities in the unit ball
			under constraints},
		Adv. Math. \textbf{459} (2024), 110023.
		
		\bibitem{ChenZhang2025}
		X. Chen and S. Zhang,
		\href{https://doi.org/10.1007/s00526-025-03095-2}
		{\emph{Constrained Moser--Trudinger--Onofri inequality
				and a uniqueness criterion for the mean field equation}},
		Calc. Var. Partial Differential Equations
		\textbf{64} (2025), Paper No. 296.
		
		\bibitem{DonskerVaradhan1975}
		M. D. Donsker and S. R. S. Varadhan,
		\href{https://doi.org/10.1002/cpa.3160280102}
		{\emph{Asymptotic evaluation of certain Markov
				process expectations for large time. I}},
		Comm. Pure Appl. Math. \textbf{28} (1975), 1--47.
		
		\bibitem{Duren1983}
		P. L. Duren,
		\href{https://link.springer.com/book/9780387907956}
		{\emph{Univalent Functions}},
		Grundlehren Math. Wiss., vol. 259, Springer-Verlag, New York, 1983.
		
		\bibitem{Fang2004}
		H. Fang,
		\href{https://doi.org/10.4310/CAG.2004.v12.n5.a8}
		{\emph{On a multi-particle Moser--Trudinger inequality}},
		Comm. Anal. Geom. \textbf{12} (2004), 1155--1171.
		
		\bibitem{FrankLieb2012}
		R. L. Frank and E. H. Lieb,
		\href{https://doi.org/10.4007/annals.2012.176.1.6}
		{\emph{Sharp constants in several inequalities
				on the Heisenberg group}},
		Ann. of Math. (2) \textbf{176} (2012), 349--381.
		
		\bibitem{GhoshKarmakar2025}
		M. Ghosh and D. Karmakar,
		\href{https://arxiv.org/abs/2508.19930}
		{\emph{Quantitative stability for the conformally
				invariant Chang--Gui inequality on the exponentiation
				of functions on the sphere}},
		Preprint (2025), arXiv:2508.19930.
		
		\bibitem{GKSV2026}
		M. Ghosh, D. Karmakar, R. Sinha and K. Verma,
		\href{https://arxiv.org/abs/2610.05356}
		{\emph{A constrained log-HLS inequality and
				its associated Wasserstein flow}},
		Preprint (2026), arXiv:2610.05356.
		
		\bibitem{GHX2026}
		C. Gui, Y. Hu and W. Xie,
		\href{https://doi.org/10.3934/dcds.2026069}
		{\emph{New sharp Moser--Trudinger type
				inequalities on the sphere}},
		Discrete Contin. Dyn. Syst. \textbf{54} (2026), 334--353.
		
		\bibitem{GLWYfull2026}
		C. Gui, T. Li, J. Wei and Z. Ye,
		\href{https://arxiv.org/abs/2608.13497}
		{\emph{A positive answer to the generalized
				Chang--Yang conjecture on $\mathbb S^N$}},
		Preprint (2026), arXiv:2608.13497.
		
		\bibitem{GHLWXY2026}
		C. Gui, Y. Hu, T. Li, J. Wei, Z. X and Z. Ye,
		\href{https://arxiv.org/abs/2610.11338.}
		{\emph{Rigidity of a $Q$-curvature-type problem on  $\mathbb S^N$ in every dimension	$N\ge 3$}},
		Preprint (2026),  arXiv:2610.11338.
		
		\bibitem{GuiMoradifam2018}
		C. Gui and A. Moradifam,
		\href{https://doi.org/10.1007/s00222-018-0820-2}
		{\emph{The sphere covering inequality and its applications}},
		Invent. Math. \textbf{214} (2018), 1169--1204.
		
		\bibitem{GuiWei2000}
		C. Gui and J. Wei,
		\href{https://doi.org/10.2140/pjm.2000.194.349}
		{\emph{On a sharp Moser--Aubin--Onofri inequality
				for functions on $S^2$ with symmetry}},
		Pacific J. Math. \textbf{194} (2000), 349--358.
		
		\bibitem{Hang2022}
		F. Hang,
		\emph{A remark on the concentration compactness principle
			in critical dimension},
		Comm. Pure Appl. Math. \textbf{75} (2022), no.~10, 2245--2278.
		
		\bibitem{HangWang2022}
		F. Hang and X. Wang,
		\href{https://doi.org/10.1093/imrn/rnab067}
		{\emph{Improved Sobolev inequality under constraints}},
		Int. Math. Res. Not. IMRN \textbf{2022} (2022),
		10822--10857.
		
		\bibitem{Ho2023}
		P. T. Ho,
		\href{https://doi.org/10.1007/s00013-023-01847-0}
		{\emph{Improved $L^p$-Folland--Stein--Sobolev
				inequality under constraints}},
		Arch. Math. (Basel) \textbf{120} (2023), 521--532.
		
		\bibitem{HuDou2023}
		Y. Y. Hu and J. B. Dou,
		\href{https://doi.org/10.1007/s10114-023-2630-8}
		{\emph{Improved Hardy--Littlewood--Sobolev inequality
				on $\mathbb S^n$ under constraints}},
		Acta Math. Sin. (Engl. Ser.)
		\textbf{39} (2023), 2149--2163.
		
		\bibitem{LebedevMilin1965}
		N. A. Lebedev and I. M. Milin, \emph{An inequality}, Vestnik Leningrad. Univ. \textbf{20} (1965), no. 19, 157--158 (in Russian).
		
		\bibitem{Lions1985LimitI}
		P.-L. Lions,
		\href{https://doi.org/10.4171/RMI/6}
		{\emph{The concentration-compactness principle
				in the calculus of variations. The limit case, Part 1}},
		Rev. Mat. Iberoam. \textbf{1} (1985),
		no. 1, 145--201.
		
		\bibitem{Moser1971}
		J. Moser,
		\href{https://iumj.org/article/2133/}
		{\emph{A sharp form of an inequality by N. Trudinger}},
		Indiana Univ. Math. J. \textbf{20} (1971), 1077--1092.
		
		\bibitem{Onofri1982}
		E. Onofri,
		\href{https://doi.org/10.1007/BF01212171}
		{\emph{On the positivity of the effective action in a theory of random surfaces}},
		Comm. Math. Phys. \textbf{86} (1982), 321--326.
		
		\bibitem{OPS1988}
		B. Osgood, R. Phillips and P. Sarnak, \emph{Extremals of determinants of Laplacians}, J. Funct. Anal. \textbf{80} (1988), 148--211.
		
		\bibitem{Rudin1980}
		W. Rudin,
		\emph{Function Theory in the Unit Ball of $\mathbb C^n$},
		Grundlehren Math. Wiss., vol. 241,
		Springer-Verlag, New York, 1980.
		
		\bibitem{Saitoh1979}
		S. Saitoh,
		\href{https://doi.org/10.1007/BF01352026}
		{\emph{Some inequalities for analytic functions
				with a finite Dirichlet integral on the unit disc}},
		Math. Ann. \textbf{246} (1979), 69--77.
		
		\bibitem{SunZou2025}
		J. Sun and W. Zou,
		\href{https://doi.org/10.1093/imrn/rnaf242}
		{\emph{Some improved almost sharp inequalities
				under constraints}},
		Int. Math. Res. Not. IMRN \textbf{2025} (2025), rnaf242.
		
		\bibitem{Villani2009}
		C. Villani,
		\href{https://doi.org/10.1007/978-3-540-71050-9}
		{\emph{Optimal Transport: Old and New}},
		Grundlehren Math. Wiss., vol. 338, Springer, Berlin, 2009.
		
		
		\bibitem{Widom1988}
		H. Widom,
		\href{https://doi.org/10.1090/S0002-9939-1988-0929019-3}
		{\emph{On an inequality of Osgood, Phillips and Sarnak}},
		Proc. Amer. Math. Soc. \textbf{102} (1988), 773--774.
		
		\bibitem{Yan2023}
		Z. Yan,
		\href{https://doi.org/10.1016/j.jfa.2023.109890}
		{\emph{Improved higher-order Sobolev inequalities
				on CR sphere}},
		J. Funct. Anal. \textbf{284} (2023), 109890.
		
		\bibitem{Yan2025}
		Z. Yan,
		\href{https://doi.org/10.1007/s00526-025-02931-9}
		{\emph{Improved Sobolev inequalities on CR sphere}},
		Calc. Var. Partial Differential Equations
		\textbf{64} (2025), Paper No. 70.
		
		\bibitem{Zygmund1959}
		A. Zygmund,
		\emph{Trigonometric Series, Vol. I},
		2nd ed., Cambridge University Press, Cambridge, 1959.
		
	\end{thebibliography}
\end{document}